\documentclass[12pt, 3p]{article}
\usepackage{amsmath,amsthm,amsfonts,amssymb}
\usepackage{fullpage}
\usepackage{subfig}
\usepackage{blkarray}
\usepackage{caption}
\usepackage{float}
\usepackage{algpseudocode}
\usepackage{tikz}
\usepackage{enumitem}

\usepackage{hyperref}
\numberwithin{equation}{section}
\usepackage{xcolor}
\usepackage{colortbl}
\usepackage[normalem]{ulem}
\usepackage{sectsty}
\usepackage{mathdots}
\definecolor{astral}{RGB}{46,116,181}
\subsectionfont{\color{astral}}
\sectionfont{\color{astral}}
\usepackage{colortbl}
\DeclareMathAlphabet{\mathpzc}{OT1}{pzc}{m}{it}
\DeclareFontFamily{OT1}{pzc}{}
\DeclareFontShape{OT1}{pzc}{m}{it}{<-> s * [0.900] pzcmi7t}{}
\DeclareMathAlphabet{\mathpzc}{OT1}{pzc}{m}{it}
\usepackage{longtable}
\usepackage{amsbsy}
\usepackage{accents}
\usepackage{arydshln}
\newlength{\dhatheight}

\newenvironment{frcseries}{\fontfamily{frc}\selectfont}{}

\newenvironment{key}
{\par\textbf{Keywords:\,}\hangindent=2.9cm\hangafter=1}
{\par}

\newenvironment{AMS}
{\par\textbf{Mathematics Subject Classification: }\hangindent=2.2cm\hangafter=1}
{\par}

\usepackage[a4paper,
  top=25mm,
  bottom=25mm,
  left=20mm,
  right=20mm
]{geometry}

\DeclareMathAlphabet\mathbfcal{OMS}{cmsy}{b}{n}
\definecolor{darkslategray}{rgb}{0.18, 0.31, 0.31}
\definecolor{warmblack}{rgb}{0.0, 0.26, 0.26}

\usepackage{multirow}

\usepackage{mathtools}
\usepackage{algorithm}
\usepackage{algorithmicx}
\makeatletter
\def\BState{\State\hskip-\ALG@thistlm}
\makeatother
\newtheorem{theorem}{Theorem}[section]
\newtheorem{lemma}[theorem]{Lemma}
\newtheorem{corollary}[theorem]{Corollary}

\theoremstyle{definition}
\newtheorem{definition}{Definition}[section]
\newtheorem{remark}{Remark}[section]
\newtheorem{example}{Example}[section]

\title{Solution Analysis of Tensor Equation $\mathcal{A} \ltimes \mathcal{X} \ltimes \mathcal{B}= \mathcal{C}$ via Semi Tensor Product with t-product}
\author{\small Bhawna Garg \thanks{Department of Mathematics, NIT Raipur,
Raipur-492010, India({\tt bhawna1601@gmail.com}). }, Ranjan Kumar Das \thanks{Corresponding author, Department of Mathematics, NIT~Raipur,
Raipur-492010, India({\tt rkdas.maths@nitrr.ac.in}). } 
\date{}}

\begin{document}
% \title{\textbf{Solution Analysis of Tensor Equation $\mathcal{A} \ltimes \mathcal{X} \ltimes \mathcal{B}= \mathcal{C}$ via Semi Tensor Product with t-product}}
\maketitle
\begin{abstract}

 % However, the equation $\mathcal{A\ltimes X=B}$ is not sufficient to describe many applications, which require the more general tensor equation $\mathcal{A\ltimes X\ltimes B=C}$, in which both left and right tensor operators are involved.
% J. Fathi {\cite{J.FathitensorequationAX=BunderSTP}} studied the solvability of the tensor equation $\mathcal{A}\ltimes \mathcal{X} = \mathcal{B}$ via the semi-tensor product with the $t$-product. But this cannot handle the more general equation $\mathcal{A}\ltimes \mathcal{X}\ltimes \mathcal{B} = \mathcal{C}$ involving both left and right tensor operators. In this manuscript, we carry out a detailed solution analysis of $\mathcal{A}\ltimes \mathcal{X}\ltimes \mathcal{B} = \mathcal{C}$ via the semi-tensor product with the $t$-product. 

Tensor equations involving both left and right tensor operators arise naturally in applications where multidimensional data are coupled through transformations acting from both sides. Motivated by the need to analyze such equations, this manuscript investigates a detailed solution analysis of $\mathcal{A}\ltimes \mathcal{X}\ltimes \mathcal{B} = \mathcal{C}$ via the semi-tensor product with the $t$-product. The manuscript provides a more general framework that includes, as a special case, the tensor equation $\mathcal{A}\ltimes\mathcal{X}=\mathcal{B}$ considered in \cite{J.FathitensorequationAX=BunderSTP}. Necessary and sufficient conditions are derived for the existence of vector and matrix valued solutions. Explicit compatibility conditions are established, and constructive Moore-Penrose-inverse-based algorithms are provided for the vector and matrix valued cases.
The computational complexity and execution times are compared for Discrete Fourier Transform(DFT) and Fast Fourier Transform(FFT) based \(t\)-product implementations. The practical relevance of the proposed framework is further demonstrated through a color image deblurring application. Examples throughout illustrate the theoretical results.
\noindent
\end{abstract}

\begin{key}Tensor Equation, Semi Tensor Product, t-product, Compatibility Conditions, Toeplitz and Circulant Tensor.
\end{key}


\section{Introduction}
Tensors provide a powerful and natural framework for representing and analyzing multiway data arising in modern scientific and engineering applications, see \cite{ J.Coopertensorapplication,T.Kolda2009Tensordecomposition, Bentbib2018, El2021} and the references therein. Unlike matrices, tensors are capable of preserving intrinsic multidimensional structures and correlations, which renders them essential in applications such as signal processing, image and video analysis, machine learning, control systems, and networked dynamical systems \cite{D.S.~Burdicktensorapplication,H.Jintensorapplication,Y.Jitensorapplication,D.mishra2020}. As the dimensionality and complexity of data continue to increase, tensor-based modeling has become increasingly important.

The algebraic manipulation of tensors relies heavily on the choice of tensor products. In recent decades, various tensor products have been introduced in the literature, including the Kronecker product \cite{K.Batselier2017Kroneckerproductoftensor}, k-mode product \cite{L.Qi2017Spectraltheory}, Einstein product \cite{Einstein2007Relativity}, Tucker product \cite{T.Kolda2009Tensordecomposition}, and the t-product \cite{M.Kilmer2011factorizationfortensor}. Among these, the t-product has gained significant attention due to its strong algebraic properties, ability to preserve matrix-like structures, and computational efficiency through block circulant representations and fast Fourier transforms. However, most existing tensor products, including the standard t-product, impose strict dimensional compatibility requirements, which limit their applicability in practical problems involving tensor dimension.

To overcome dimensional incompatibility, Cheng et al. \cite{D.cheng2012STPofmatrices} introduced the \emph{semi-tensor product (STP)}, which generalizes the conventional matrix product by embedding matrices into higher-dimensional spaces via the Kronecker product. The STP coincides with the standard matrix multiplication when dimensions are compatible, while remaining well defined for arbitrary dimensions. Due to its flexibility, the STP has become an important algebraic tool in control theory, Boolean networks, game theory, and nonlinear dynamical systems, and has been successfully extended from matrices to higher-order tensors \cite{J.FathitensorequationAX=BunderSTP,J.2017STPapplication,W.Liu2022STPoftensor}.

Matrix equations of the form
$AX=C$
play a fundamental role in image deblurring, systems theory and numerical analysis \cite{J.FathitensorequationAX=BunderSTP, STP2023, J.Yao2016solutionAX=B}. Under the STP framework, this classical equation has been extended to cases with incompatible dimensions, and further generalized to tensor equations, allowing the unknown variable $X$ to be a vector, matrix, or tensor form. These developments have significantly broadened the scope of solvable linear systems in multilinear algebra \cite{J.FathitensorequationAX=BunderSTP}.
\par Beyond the equation $AX=C,$ more general matrix equation of the form
$AXB=C$ have also been extensively studied by employing the STP approach. Several studies have investigated existence conditions and computational approaches for the matrix equation
$A \ltimes X \ltimes B = C$, where $\ltimes$ is referred to as semi-tensor product and the matrices $A$, $B$, $C$ are given and the matrix $X$ is to be determined \cite{Z.Dhong2019AXB=CunderSTP}. These results demonstrate the effectiveness of STP.

Motivated by the matrix equation $ A \ltimes X \ltimes B = C $, this paper extends the framework to the tensor setting by integrating the semi-tensor product (STP) with the t-product. We investigate the tensor equation $\mathcal{A} \ltimes \mathcal{X} \ltimes \mathcal{B} = \mathcal{C}$, where $\mathcal{A}$, $\mathcal{B}$, and $\mathcal{C}$ are third-order tensors and the unknown $\mathcal{X}$ may be vector, matrix, or higher-order tensor valued. The proposed framework is applicable to tensors of arbitrary sizes due to the dimensional flexibility of the STP. The structure of the paper is as follows: Section~2 introduces the necessary definitions and notation, including a Fourier-domain characterization of the $t$-product. The tensor vector case is studied in Section~3, where compatibility conditions for solvability, an explicit solution procedure, an accompanying algorithm, and the structure of $\mathcal{C}$ are established. Section~4 extends these results to the tensor matrix case and presents an algorithm for obtaining the solution, while Section~5 further generalizes the analysis to the tensor case. The computational complexity of the proposed framework is investigated in Section~6, with theoretical and numerical comparisons between DFT- and FFT-based implementations of the $t$-product. Finally, Section~7 demonstrates the practical applicability of the proposed framework through a color image deblurring problem.

%%%%%%%%%%%%%%%%%%%%%%%%%%%%%%%%%%%%%%%%%%%%%%%%%%%%%%%%%%%%%%%%%%%%%%%%%%%%%%%%%%%%%%%%%%%%%%%%%%%%%%
\section{Preliminaries}
In this section, we briefly introduce the notation and fundamental concepts used throughout the paper. Unless otherwise stated, all symbols are employed consistently. The sets $\mathbb{N},$  $\mathbb{R}$ and $\mathbb{C}$ denote the set of positive integers, real and complex numbers respectively, while $\mathbb{R}^n  (\text{resp.} \ \mathbb{C}^n)$ represents the set of real (complex) column vectors of dimension n, and $\mathbb{R}^{m\times n} (\text{resp.}\ \mathbb{C}^{m\times n})$ denotes the space of real (complex) matrices of size  $m\times n$. Uppercase letters are used to denote matrices, italicized uppercase letters denote vectors, lowercase letters denote scalars, and these conventions naturally extend to block structures and higher-order objects.The $(i,j)^{th}$ element of a matrix $A$ is denoted by $a_{ij}$, and $A_i$ represents its $i^{th}$ column. For positive integers $m$ and $n$, $\operatorname{lcm}(m,n)$ and $\operatorname{gcd}(m,n)$ denote their least common multiple and greatest common divisor, respectively. Throughout this paper, the rank of a matrix (or tensor unfolding) is denoted by $\rho$.
A tensor is a multidimensional array, and its order-also called the number of modes corresponds to the number of indices required to specify an element. Scalars, vectors, and matrices can be viewed as tensors of order zero, one, and two, respectively. \cite{J.FathitensorequationAX=BunderSTP} For an $N^{th}$ order tensor $\mathcal{A} \in \mathbb{R}^{n_1 \times n_2 \times \cdots \times n_N}$, the notation $a_{i_1 i_2 \cdots i_N}$, where $1 \leq i_{j}\leq n_{j}$ for $j = 1, \ldots, N,$ denotes its $(i_1, i_2, \ldots, i_N)^{th}$ entry.
Given a tensor $\mathcal{A} \in \mathbb{R}^{n_1 \times n_2 \times \cdots \times n_N}$, fixing the last index 
$ \mathcal{A}_{: \, : \, \cdots \, : \, k}, \ k = 1,2,\ldots,n_N,$
which lies in $\mathbb{R}^{n_1 \times n_2 \times \cdots \times n_{N-1}}$ and is referred to as a \emph{frontal slice}.\\
More generally, \emph{fibers} are obtained by fixing all but one index of the tensor, serving as the higher-order generalization of matrix rows and columns.\\
\begin{figure}[htbp]
    \centering
    \includegraphics[width=0.4\textwidth]{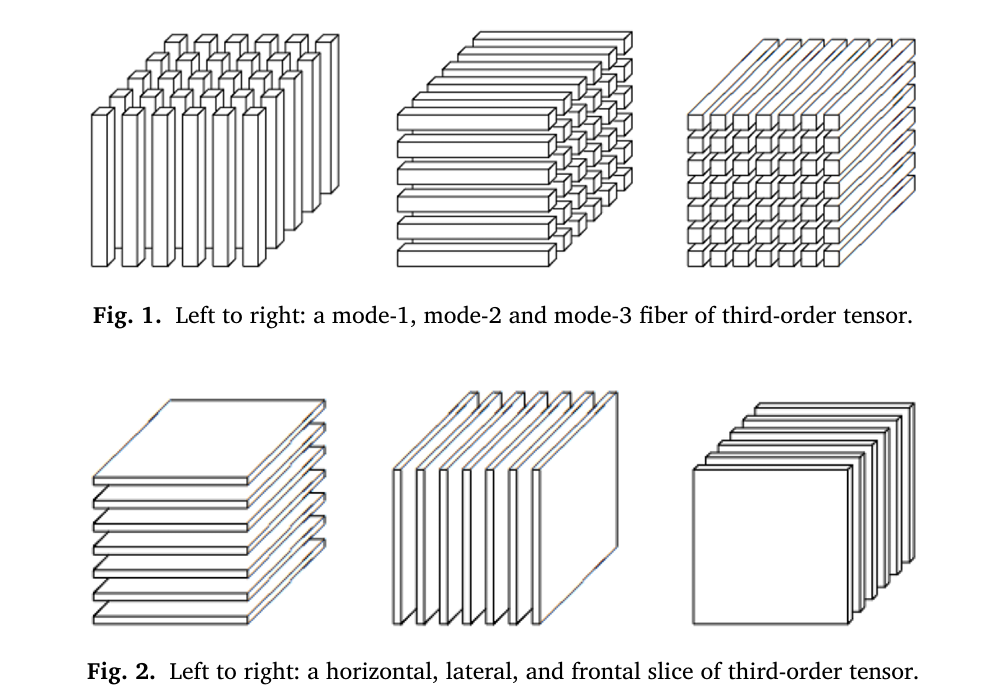}
    \label{fig:example}
\end{figure}\\
We consider third-order tensors ($N=3$), denoted by calligraphic letters.
A tensor $\mathcal{A} \in \mathbb{C}^{n_1 \times n_2 \times n_3}$ is written as
$\mathcal{A}=(a_{ijk})_{i,j,k=1}^{n_1,n_2,n_3}$.
The notations $\mathcal{A}{(i:k)}$, $\mathcal{A}{(:jk)}$, and $\mathcal{A}{(ij:)}$ represent the fibers of the tensor $\mathcal{A}$ along the first, second, and third modes, respectively, see Fig.1.
Likewise, $\mathcal{A}{(i::)}$, $\mathcal{A}{(:j:)}$, and $\mathcal{A}{(::k)}$
respectively the $i^{th}$ horizontal, $j^{th}$ lateral, and $k^{th}$ frontal slices of $\mathcal{A}$,
each of which is a matrix, see Fig.2. A tensor multiplication scheme is developed by viewing a third-order tensor as a stack of frontal matrices. Specifically, for $\mathcal{A} \in \mathbb{R}^{m \times n \times r}$, the tensor is composed of $r$ frontal slices of dimension $m \times n$, denoted by $\mathcal{A}{(::1)}, \mathcal{A}{(::2)}, \ldots, \mathcal{A}{(::r)}.$
%%%%%%%%%%%%%%%%%%%%%%%%%%%%%%%%%%%%%%%%%%%%%%%%%%%%%%%%%%%%%%%%%%%%%%%%%%%%%%%%%%%%%%%%%%%%%%%%%%%%%%%%%%%%%%%%%%%%
\begin{definition}[t-product  \cite{ M.Kilmer2008SVDtensor,M.Kilmer2011factorizationfortensor}]\label{def:tproduct}
Let $\mathcal{A}\in\mathbb{C}^{m\times n\times r}$ and $\mathcal{B}\in\mathbb{C}^{n\times s\times r}$ be third-order tensors. The \emph{t-product} of $\mathcal{A}$ and $\mathcal{B}$, denoted by $\mathcal{A}*\mathcal{B}$, is the $m\times s\times r$ tensor
\[
\mathcal{A}*\mathcal{B} := \mathrm{fold}\big(\mathrm{bcirc}(\mathcal{A})\cdot \mathrm{unfold}(\mathcal{B})\big),
\]
where
\[
\mathrm{unfold}(\mathcal{A}) =
\begin{bmatrix}
\mathcal{A}_{::1}\\ \mathcal{A}_{::2}\\ \vdots\\ \mathcal{A}_{::r}
\end{bmatrix}\in\mathbb{C}^{mr\times n},
\
\mathrm{bcirc}(\mathcal{A}) =
\begin{bmatrix}
\mathcal{A}_{::1} & \mathcal{A}_{::r} & \cdots & \mathcal{A}_{::2}\\
\mathcal{A}_{::2} & \mathcal{A}_{::1} & \cdots & \mathcal{A}_{::3}\\
\vdots & \vdots & \ddots & \vdots\\
\mathcal{A}_{::r} & \mathcal{A}_{::r-1} & \cdots & \mathcal{A}_{::1}
\end{bmatrix},
\]
and $\mathrm{fold}$ is the inverse process of $\mathrm{unfold}$. Computing the t-product directly from Definition~\ref{def:tproduct} costs $O(m n r^2 s)$ flops, since $\mathrm{bcirc}(\mathcal{A})$ is an $mr\times nr$ matrix. This cost can be reduced substantially by exploiting the fact that $\mathrm{bcirc}(\mathcal{A})$ is block-circulant, and every block-circulant matrix is block-diagonalized by the (unnormalized) discrete Fourier transform. Precisely, letting $F_r\in\mathbb{C}^{r\times r}$ denote the DFT matrix and $F_r^{*}$ its conjugate transpose, we have \[ \big(F_r\otimes I_m\big)\,\mathrm{bcirc}(\mathcal{A})\,\big(F_r^{*}\otimes I_n\big) = \mathrm{blkdiag}\big(\widehat{\mathcal{A}}_{::1},\,\widehat{\mathcal{A}}_{::2},\,\ldots,\,\widehat{\mathcal{A}}_{::r}\big), \] where $\widehat{\mathcal{A}} := \mathcal{A}\times_3 F_r$ is obtained by applying the length-$r$ discrete Fourier transform along every mode-3 fiber (tube) of $\mathcal{A}$. Applying the same transform to $\mathcal{B}$ along mode 3 gives $\widehat{\mathcal{B}}$, and the t-product reduces to $r$ independent matrix products of the transformed frontal slices, followed by an inverse transform: \[ \mathcal{A}*\mathcal{B} = \widehat{\mathcal{C}}\times_3 F_r^{-1}, \qquad \widehat{\mathcal{C}}_{::i} = \widehat{\mathcal{A}}_{::i}\,\widehat{\mathcal{B}}_{::i},\quad i=1,\ldots,r. \] Each length-$r$ Fourier transform along a tube can be computed in $O(r\log_2 r)$ flops using the FFT, rather than the $O(r^2)$ flops required by explicit multiplication with $F_r$. in the following section This section compares the computational performance of the proposed STP-based $t$-product using the direct DFT and FFT approaches. The comparison is carried out for vector and matrix solutions in terms of computational time and theoretical computational cost.
\end{definition}
\begin{example}
Let tensors $\mathcal{A}\in\mathbb{R}^{2\times2\times2}$ and
$\mathcal{B}\in\mathbb{R}^{2\times2\times2}$ be given as follows:
[-2em]
\begin{center}
\begin{tabular}{cc|cc}
\hline
\multicolumn{2}{c}{${\mathcal A}(:,:,1)$} &
\multicolumn{2}{c}{${\mathcal A}(:,:,2)$} \\
\hline
1&2&5&6\\
3&4&7&8
\end{tabular}
\qquad
\begin{tabular}{cc|cc}
\hline
\multicolumn{2}{c}{${\mathcal B}(:,:,1)$} &
\multicolumn{2}{c}{${\mathcal B}(:,:,2)$} \\
\hline
1&0&1&1\\
0&1&0&1
\end{tabular}
\end{center}

Using the definition of the $t$-product,$
\mathcal{A}*\mathcal{B}
=
\operatorname{fold}\left(
\operatorname{bcirc}(\mathcal{A})
\operatorname{unfold}(\mathcal{B})
\right),
$\\
where $$
\operatorname{bcirc}(\mathcal{A})
=
\begin{bmatrix}
1&2&5&6\\
3&4&7&8\\
5&6&1&2\\
7&8&3&4
\end{bmatrix}
\text{and}
\operatorname{unfold}(\mathcal{B})
=
\begin{bmatrix}
1&0\\
0&1\\
1&1\\
0&1
\end{bmatrix}.
$$
Therefore, $$
\begin{aligned}
\mathcal{A}*\mathcal{B}
&=
\operatorname{fold}
\left(
\begin{bmatrix}
6&13\\
10&19\\
6&9\\
10&15
\end{bmatrix}
\right).
\end{aligned}
$$
Equivalently, the frontal slices of $\mathcal{C}=\mathcal{A}*\mathcal{B}$ are
\begin{tabular}{cc|cc}
\hline
\multicolumn{2}{c}{${\mathcal C}(:,:,1)$} &
\multicolumn{2}{c}{${\mathcal C}(:,:,2)$} \\
\hline
6&13&6&9\\
10&19&10&15
\end{tabular}
\end{example}

\begin{definition}(Kronecker Product \cite{K.Batselier2017Kroneckerproductoftensor, W.Liu2022STPoftensor})
Let $\mathcal{A}=(a_{i_1 i_2 i_3})\in\mathbb{R}^{n_1\times n_2\times n_3}$
and
$\mathcal{B}=(b_{j_1 j_2 j_3})\in\mathbb{R}^{p_1\times p_2\times p_3}$
be two $3^{rd}$ tensors.
The \emph{Kronecker product} of $\mathcal{A}$ and $\mathcal{B}$, denoted by
$\mathcal{A}\otimes\mathcal{B}$, is defined as an $3^{rd}$ order tensor of size
$n_1p_1\times n_2p_2\times n_3p_3$ given by
\[
\mathcal{A}\otimes\mathcal{B}
:=
\bigl(a_{i_1 i_2 i_3}\,\mathcal{B}\bigr).
\]
More explicitly, its entries satisfy
$(\mathcal{A}\otimes\mathcal{B})_{\ell_1\,\ell_2\ell_3}
=
a_{i_1 i_2 i_3}\,
b_{j_1 j_2  j_3},$
here, each index $\ell_k$ represents a linear ordering of the pair $(i_k,j_k)$ defined by
\[
\ell_k=(i_k-1)p_k+j_k,\quad 1\le k\le 3,
\]
with $p_k$ being the size of the $k^{th}$ mode of tensor $\mathcal{B}$.
\end{definition}

\begin{definition}(Identity Tensor {\cite{Z.Chen2023SVDlikedecomposition}})
An $n \times n \times k$ third-order tensor $\mathcal{I}_{nnk}$ is referred to as an
\emph{identity tensor} if its first frontal slice equals the $n \times n$
identity matrix and all remaining frontal slices are zero matrices.
In the special case $k=1$, the tensor $\mathcal{I}_{nn1}$ consists of a single
frontal slice that is the $n \times n$ identity matrix. For simplicity, this
tensor is denoted by $\mathcal{I}_n$.
\end{definition}

The STP preserves the
fundamental algebraic properties of the $t$-product while extending its
applicability to tensors whose dimensions are not necessarily compatible.
As a result, the STP serves as a natural generalization of
the $t$-product.
\begin{definition}(Semi Tensor Product of third-order Tensors {\cite{J.FathitensorequationAX=BunderSTP}})
Let $\mathcal{A}=(a_{i_1 i_2 i_3}) \in \mathbb{C}^{m \times n \times r}$ and
$\mathcal{B}=(b_{j_1 j_2 j_3}) \in \mathbb{C}^{p \times q \times s}$ be two
third-order tensors.
The \emph{STP} of $\mathcal{A}$ and $\mathcal{B}$,
denoted by $\mathcal{A} \ltimes \mathcal{B}$, is defined as
\[
\mathcal{A} \ltimes \mathcal{B}
:=
\bigl(
\mathcal{A} \otimes \mathcal{I}_{\frac{l_{1}}{n} \times \frac{l_{1}}{n} \times \frac{l_{2}}{r}}
\bigr)
*
\bigl(
\mathcal{B} \otimes \mathcal{I}_{\frac{l_{1}}{p} \times \frac{l_1}{p} \times \frac{l_{2}}{s}}
\bigr)
\in \mathbb{C}^{\frac{ml_{1}}{n} \times \frac{ql_{1}}{p} \times l_{2}},
\]
where $l_{1}=\operatorname{lcm}(n,p)$ and $l_{2}=\operatorname{lcm}(r,s)$.
\end{definition}
Note that when $n=p$ and $r=s$, the STP reduces to the usual
$t$-product of third-order tensors. Moreover, in the special case of second-order tensors, this construction reduces to the usual STP defined for matrices.
%%%%%%%%%%%%%%%%%%%%%%%%%%%%%%%%%%%%%%%%%%%%%%%%%%%%%%%%%%%%%%%%%%%%%%%%%%%%%%%%%%%%%%%%%%%%%%%%%%%%%%%%

\section{Solvability analysis for tensor equation with vector solutions}
Here we examine solvability properties of tensor–vector equations formulated through the STP:
\begin{equation}\label{main equation}
\underbrace{\underbrace{\mathcal{A} \ltimes X}_{\frac{m l_1}{n} \times \frac{l_1}{p} \times l_3}\ltimes
\mathcal{B}}_{\frac{m l_2 p}{n} \times \frac{b l_2}{a} \times l_4}  = \mathcal{C} \in \mathbb{C}^{h \times k \times r},
\end{equation}
here $\mathcal{A} \in \mathbb{C}^{m \times n \times r}$,
$\mathcal{B} \in \mathbb{C}^{a \times b \times r}$, 
$\mathcal{C} \in \mathbb{C}^{h \times k \times r}$ are known tensors, $X \in \mathbb{C}^{p}$ is the unknown vector to be determined and $l_3=\operatorname{lcm}\{r,1\}, \  
l_1 = \operatorname{lcm}\{n,p\},\ 
l_2 = \operatorname{lcm}\{\frac{l_1}{p}, a\},
\  l_4=\operatorname{lcm}\{l_3,r\}.$\\

We first analyze the special case $h=m$ then we will extend the discussion with unrestricted dimensions case.
\subsection{The special case corresponds to \texorpdfstring{$h= m$}{h = m}}
Here we investigate the compatibility of Eq.(\ref{main equation}) when $h= m$. We first present a lemma that provides necessary conditions on the tensor dimensions for the Eq.(\ref{main equation}) to admit a solution, along with constraints on the size of the solution vector $X\in \mathbb{C}^{p}$.
\begin{lemma} \label{m=h vector case}
The existence of a solution vector of size $p$ for Eq.(\ref{main equation}) with $h=m$ imposes the following necessary dimensional conditions on $\mathcal{A}$, $\mathcal{B}$, $\mathcal{C}$ and $p$: \vspace{-0.4em}
\begin{enumerate}
\item[(i)] $\frac{k}{b}, \frac{n}{a}$ must belong to $\mathbb{N}$ and $\frac{k}{b} \mid \frac{n}{a}$;
~ (ii) $p = \frac{bn}{ak}$.
\end{enumerate}
\end{lemma}
\begin{proof}
Let $X \in \mathbb{C}^{p}$ satisfies Eq.(\ref{main equation}).  Then by Eq.(\ref{main equation}) dimension consistency implies
\[
\frac{m l_2 p}{n}=h, \qquad
\frac{b l_2}{a}=k, \qquad
l_4=r.
\]
It follows directly that $p=\frac{bn}{ak}
\quad \text{and} \quad
\frac{k}{b}=\frac{l_2}{a},$ which requires $\frac{k}{b}\in\mathbb{N}$.
Moreover, $l_2=\frac{ak}{b}
=\operatorname{lcm}\left\{\frac{ak l_1}{bn},\, a\right\}
=\frac{ak l_1}{bn},$ yielding $l_1=n$. Therefore, $\frac{n}{a}=p\,\frac{k}{b}\in\mathbb{N}.$
\end{proof}
Lemma~\ref{m=h vector case} provides dimensional requirements that must be satisfied for Eq.(\ref{main equation}) to admit a solution. These constraints are known as \emph{admissibility conditions}; when they are fulfilled, the tensors $\mathcal{A}$, $\mathcal{B}$, and $\mathcal{C}$ are considered to be admissible with one another.

 %%%%%%%%%%%%%%%%%%%%%%%%%%%%%%%%%%%%%%%%%%%%%%%%%%%%%%%%%%%%%%%%%%%%%%%%%%%%%%%%%%%%%%%%
 \begin{definition}(Column Stacking \cite{D.cheng2012STPofmatrices})
Let $A=(a_{ij})\in \mathbb{C}^{m \times n}.$ The column stacking form of $A,$ denoted by ${V}_c(A),$ is defined as 
$$ {V}_c(A)= (a_{11},\ldots,a_{m1},a_{12},\ldots,a_{m2},\ldots,a_{1n},\ldots,a_{mn})^T.$$
\end{definition}
The column-stacking operation is analogously defined for tensors as follows.
\begin{definition}(Lateral form of tensor)
Let \(\mathcal{A}\in\mathbb{C}^{n_1\times n_2\times n_3}\). The lateral form of \(\mathcal{A}\), denoted by \(\mathcal{V}_L(\mathcal{A})\), is a vector in \(\mathbb{C}^{n_1n_2n_3\times 1}\), defined as follows:

$$\mathcal{V}_L(\mathcal{A}):=[{V}_c(\mathcal{A}(:1:)), {V}_c(\mathcal{A}(:2:)), \ldots,{V}_c(\mathcal{A}(:n_2:)) ]^T.$$
\end{definition}
\begin{example} Let us consider $\mathcal{A}$ as follows:
$$\begin{tabular}{cc|cc}
 \hline
\multicolumn{2}{c}{${{\mathcal{A}}}(:,:,1)$} & 
\multicolumn{2}{c}{${\mathcal{A}}(:,:,2)$}  \\
\hline
 1&2&5&6 \\
 3&4&7&8\\
 \hline
 \end{tabular}, $$ then the lateral form of $\mathcal{A}$ is $\mathcal{V}_L(\mathcal{A})=[1,3,5,7,2,4,6,8]^T.$ 
\end{example}
We now establish necessary and sufficient criteria for the existence of solutions to Eq.(\ref{main equation}) under the assumption $h=m$. These conditions are obtained by analyzing the block structure naturally arising from the semi-tensor product. Let $\mathcal{A}_j$ represent the $j^{th}$ lateral slice of $\mathcal{A}$ and $\mathcal{B}_i$ the $i^{th}$ horizontal slice of $\mathcal{B}$. Decompose the tensor $\mathcal{A}$ as $\mathcal{A} = [\, \widehat{\mathcal{A}}_1 : \widehat{\mathcal{A}}_2 : \cdots : \widehat{\mathcal{A}}_p \,],$ here each sub-tensor $\widehat{\mathcal{A}}_j \in \mathbb{C}^{m \times \frac{ak}{b} \times r}$ for $j=1,\ldots,p$. Write $X=(x_1,x_2,\ldots,x_p)^{T}\in\mathbb{C}^p$. By applying the definition of the STP and using the dimensional relationships given in Lemma~\ref{m=h vector case}, Eq.~(\ref{main equation}) can be expressed in an equivalent form as follows:
 \begin{align*}
\mathcal{A} \ltimes X \ltimes \mathcal{B}
&= (\mathcal{A} \otimes \mathcal{I}_{\frac{l_1}{n} \times \frac{l_1}{n} \times \frac{l_3}{r}}) (X \otimes \mathcal{I}_{\frac{l_1}{p} \times \frac{l_1}{p} \times \frac{l_3}{1}}) \ltimes \mathcal{B} \\
&= [\, \widehat{\mathcal{A}}_1 : \widehat{\mathcal{A}}_2 : \cdots : \widehat{\mathcal{A}}_p \,]
   (X \otimes \mathcal{I}_{\frac{l_1}{p} \times \frac{l_1}{p} \times r}) \ltimes \mathcal{B} \\
&= x_1 \widehat{\mathcal{A}}_1 \ltimes \mathcal{B} + x_2 \widehat{\mathcal{A}}_2 \ltimes \mathcal{B} + \cdots
   + x_p \widehat{\mathcal{A}}_p \ltimes \mathcal{B} \\
&= x_1(\widehat{\mathcal{A}}_1 \otimes \mathcal{I}_{\frac{l_2b}{ak} \times \frac{l_2b}{ak} \times 1})
   (\mathcal{B} \otimes \mathcal{I}_{\frac{l_2}{a} \times \frac{l_2}{a} \times 1})
   + x_2(\widehat{\mathcal{A}}_2 \otimes \mathcal{I}_{\frac{l_2b}{ak} \times \frac{l_2b}{ak} \times 1})
   (\mathcal{B} \otimes \mathcal{I}_{\frac{l_2}{a} \times \frac{l_2}{a} \times 1}) \\
&\quad + \cdots
   + x_p(\widehat{\mathcal{A}}_p \otimes \mathcal{I}_{\frac{l_2b}{ak} \times \frac{l_2b}{ak} \times 1})
   (\mathcal{B} \otimes \mathcal{I}_{\frac{l_2}{a} \times \frac{l_2}{a} \times 1}) \\
&= x_1 \widehat{\mathcal{A}}_1 (\mathcal{B} \otimes \mathcal{I}_{\frac{k}{b} \times \frac{k}{b} \times 1})
   + x_2 \widehat{\mathcal{A}}_2 (\mathcal{B} \otimes \mathcal{I}_{\frac{k}{b} \times \frac{k}{b} \times 1})
   + \cdots
   + x_p \widehat{\mathcal{A}}_p (\mathcal{B} \otimes \mathcal{I}_{\frac{k}{b} \times \frac{k}{b} \times 1}) \\
   \end{align*}
   \begin{align*}
&= x_1 [\mathcal{A}_1 \ \cdots \ \mathcal{A}_\frac{k}{b}](\mathcal{B}_1 \otimes \mathcal{I}_{\frac{k}{b} \times \frac{k}{b} \times 1})
   + \cdots
   + x_1 [\mathcal{A}_{\frac{(a-1)k}{b}+1} \ \cdots \ \mathcal{A}_{\frac{ak}{b}}]
   (\mathcal{B}_a \otimes \mathcal{I}_{\frac{k}{b} \times \frac{k}{b} \times 1}) \\
&\quad + x_2 [\mathcal{A}_{\frac{ak}{b}+1} \ \cdots \ \mathcal{A}_{\frac{(a+1)k}{b}}]
   (\mathcal{B}_1 \otimes \mathcal{I}_{\frac{k}{b} \times \frac{k}{b} \times 1})
   + \cdots \\
&\quad + x_p [\mathcal{A}_{\frac{(pa-1)k}{b}+1} \ \cdots \ \mathcal{A}_{\frac{pak}{b}}]
   (\mathcal{B}_a \otimes \mathcal{I}_{\frac{k}{b} \times \frac{k}{b} \times 1}) \\
&= \mathcal{C}.
\end{align*}
Under this representation, Eq.(\ref{main equation}) can be rewritten as a linear combination of tensor blocks in the form
\begin{equation}\label{eq:2}
\sum_{t=1}^{p} x_t \, \widetilde{\mathcal{A}}_t = \mathcal{C},
\end{equation}
where, for each $t = 1,\ldots,p$, the tensor $\widetilde{\mathcal{A}}_t$ is defined by
\[
\widetilde{\mathcal{A}}_t
=
\sum_{l=1}^{a}
\big[
\mathcal{A}_{\frac{(t-1)a k}{b}+(l-1)\frac{k}{b}+1}
\ \cdots \
\mathcal{A}_{\frac{(t-1)a k}{b}+l\frac{k}{b}}
\big]
(\mathcal{B}_l \otimes I_{\frac{k}{b}}).
\]

After applying the lateral form of tensor $\mathcal{V}_L(\cdot)$, the existence of a solution to Eq.(\ref{main equation}) in the case $h=m$ is reduced to determining whether a corresponding linear system is consistent:
\[
\sum_{t=1}^{p} x_t \, \mathcal{V}_L(\widetilde{\mathcal{A}}_t)
=
\begin{bmatrix}
\mathcal{V}_L(\widetilde{\mathcal{A}}_1) &
\mathcal{V}_L(\widetilde{\mathcal{A}}_2) &
\cdots &
\mathcal{V}_L(\widetilde{\mathcal{A}}_p)
\end{bmatrix}
X
=
\mathcal{V}_L(\mathcal{C}).
\]
%%%%%%%%%%%%%%%%%%%%%%%%%%%%%%%%%%%%%%%%%%%%%%%%%%%%%%%%%%%%%%%%%%%
The following result follows from the above discussion.
\begin{corollary} \label{cor3}
For the case $m=h$, the tensor vector Eq.(\ref{main equation}) admits a solution if and only if
\[
\rho
\big[
\mathcal{V}_L(\widetilde{\mathcal{A}}_1)\;
\mathcal{V}_L(\widetilde{\mathcal{A}}_2)\;
\cdots\;
\mathcal{V}_L(\widetilde{\mathcal{A}}_p)
\big]
=
\rho
\big[
\mathcal{V}_L(\widetilde{\mathcal{A}}_1)\;
\mathcal{V}_L(\widetilde{\mathcal{A}}_2)\;
\cdots\;
\mathcal{V}_L(\widetilde{\mathcal{A}}_p)\;
\mathcal{V}_L(\mathcal{C})
\big],\ \text{where $\rho$ is rank of matrix.}
\]
\end{corollary}
\begin{remark}\label{remark3.1}
In the case \(m=h\), define $
M=\big[\,V_L(\widetilde{\mathcal A}_1)\ \cdots\ V_L(\widetilde{\mathcal A}_p)\,\big],
$ as above, and let \(M^\dagger\) denote the Moore-Penrose inverse of \(M\). Then, the following results hold:
\begin{enumerate}
\item Eq.~\eqref{main equation} is solvable if and only if
$$M M^\dagger V_L(\mathcal{C}) = V_L(\mathcal{C}).$$ 
\item When $(1)$ holds, the complete solution set of Eq.~\eqref{main equation} is
$$X = M^\dagger V_L(\mathcal{C}) + \big(I_p - M^\dagger M\big)\,y,
\quad y \in \mathbb{C}^p \text{ arbitrary}.$$
\end{enumerate}
\end{remark}
\begin{algorithm}\label{algorithm1}
\caption{Vector Solution of $\mathcal{A}\ltimes X\ltimes\mathcal{B}=\mathcal{C}$ (case $m=h$)}
\label{alg:vector_solution}
\begin{algorithmic}[1]
\State \textbf{
} $\mathcal{A}\in\mathbb{C}^{m\times n\times r}$, $\mathcal{B}\in\mathbb{C}^{a\times b\times r}$, $\mathcal{C}\in\mathbb{C}^{m\times k\times r}$
\State Check admissibility condition (Lemma~\ref{m=h vector case}): verify $\dfrac{k}{b},\dfrac{n}{a}\in\mathbb{N}$ and $\dfrac{k}{b}\ \Big|\ \dfrac{n}{a}$.
\If{admissibility fails}
    \State \textbf{Output:} ``Eq.~(\ref{main equation}) is inadmissible -no solution.''
\EndIf
\State Compute $p \leftarrow \dfrac{bn}{ak}$.
\For{$t = 1,\ldots,p$}
    \State Compute $\widetilde{\mathcal{A}}_t = \displaystyle\sum_{l=1}^{a}
        \Big[\mathcal{A}_{(t-1)\frac{ak}{b}+(l-1)\frac{k}{b}+1}\ \cdots\ \mathcal{A}_{(t-1)\frac{ak}{b}+l\frac{k}{b}}\Big]
        \big(\mathcal{B}_l\otimes I_{k/b}\big)$
\EndFor
\State form $M \leftarrow \big[\,V_L(\widetilde{\mathcal{A}}_1)\ \ V_L(\widetilde{\mathcal{A}}_2)\ \cdots\ V_L(\widetilde{\mathcal{A}}_p)\,\big]$.
\State Compute the Moore-Penrose inverse $M^{\dagger}$.
\State Compute residual $\mathrm{res} \leftarrow MM^{\dagger}V_L(\mathcal{C}) - V_L(\mathcal{C})$.
\If{$\|\mathrm{res}\| > 0$}
    \State \textbf{Output:} ``Eq.~(\ref{main equation}) has no solution.''
\Else
    \State Compute $X \leftarrow M^{\dagger}V_L(\mathcal{C}) + \big(I_p - M^{\dagger}M\big)y,\quad y\in\mathbb{C}^p$ arbitrary.
    \If{$\mathrm{rank}(M) = p$}
        \State \textbf{Output:} unique solution $X = M^{\dagger}V_L(\mathcal{C})$.
    \Else
        \State \textbf{Output:} general solution $X = M^\dagger V_L(\mathcal{C}) + \big(I_p - M^\dagger M\big)\,y,
\quad y \in \mathbb{C}^p \text{ arbitrary}.$.
    \EndIf
\EndIf
\end{algorithmic}
\end{algorithm}
Next, we provide several examples illustrating the above Lemma \ref{m=h vector case} and Remark \ref{remark3.1}.

\begin{example}
Consider the tensors
$\mathcal{A}\in\mathbb{C}^{3\times 6\times 2}$,
$\mathcal{B}\in\mathbb{C}^{1\times 2\times 2}$, and
$\mathcal{C}\in\mathbb{C}^{3\times 4\times 2}$ given by
$$
\begin{tabular}{cccccc|cccccc}
    \hline
    \multicolumn{6}{c}{$\mathcal{A}(:,:,1)$} &
    \multicolumn{6}{c}{$\mathcal{A}(:,:,2)$} \\
    \hline
    -3&-3&0&0&-1&0&
     2&3&-1&-3&-2&2\\
    -1&1&2&3&-1&1&
     1&1&0&1&0&3\\
    -1&-1&0&-3&3&0&
    -2&1&2&1&3&-2\\
    \hline
\end{tabular},\qquad
\begin{tabular}{cc|cc}
    \hline
    \multicolumn{2}{c}{$\mathcal{B}(:,:,1)$} &
    \multicolumn{2}{c}{$\mathcal{B}(:,:,2)$} \\
    \hline
    -1&1&-3&3\\
    \hline
\end{tabular}
$$ and
$$
\begin{tabular}{cccc|cccc}
    \hline
    \multicolumn{4}{c}{$\mathcal{C}(:,:,1)$} &
    \multicolumn{4}{c}{$\mathcal{C}(:,:,2)$} \\
    \hline
    12&-39&-12&39&
    28&3&-28&-3\\
    1&-29&-1&29&
    19&-7&-19&7\\
    -16&14&16&-14&
    -24&2&24&-2\\
    \hline
\end{tabular}.
$$
Following the construction used in the derivation preceding
Theorem~3.2, the tensor $\mathcal{A}$ is partitioned into
$p=3$ sub-blocks
$\widetilde{\mathcal{A}}_1,\widetilde{\mathcal{A}}_2,
\widetilde{\mathcal{A}}_3\in\mathbb{C}^{3\times2\times2}$.
These sub-blocks are
$$
\begin{tabular}{cccc|cccc}
    \hline
    \multicolumn{4}{c}{$\widetilde{\mathcal{A}}_1(:,:,1)$} &
    \multicolumn{4}{c}{$\widetilde{\mathcal{A}}_1(:,:,2)$} \\
    \hline
    -3&-6&3&6&
     7&6&-7&-6\\
    -2&-4&2&4&
     2&-4&-2&4\\
     7&-2&-7&2&
     5&2&-5&-2\\
    \hline
\end{tabular},
\qquad
\begin{tabular}{cccc|cccc}
    \hline
    \multicolumn{4}{c}{$\widetilde{\mathcal{A}}_2(:,:,1)$} &
    \multicolumn{4}{c}{$\widetilde{\mathcal{A}}_2(:,:,2)$} \\
    \hline
     3&9&-3&-9&
     1&3&-1&-3\\
    -2&-6&2&6&
    -6&-10&6&10\\
    -6&0&6&0&
    -2&8&2&-8\\
    \hline
\end{tabular},
$$

$$
\begin{tabular}{cccc|cccc}
    \hline
    \multicolumn{4}{c}{$\widetilde{\mathcal{A}}_3(:,:,1)$} &
    \multicolumn{4}{c}{$\widetilde{\mathcal{A}}_3(:,:,2)$} \\
    \hline
     7&-6&-7&6&
     5&-2&-5&2\\
     1&-9&-1&9&
     3&-3&-3&3\\
    -12&6&12&-6&
    -12&2&12&-2\\
    \hline
\end{tabular}.
$$

Using the vectorization operator $V_L$, define

$$
M=
\begin{bmatrix}
V_L(\widetilde{\mathcal{A}}_1)&
V_L(\widetilde{\mathcal{A}}_2)&
V_L(\widetilde{\mathcal{A}}_3)
\end{bmatrix}
\in\mathbb{C}^{24\times3},
\
c=V_L(\mathcal{C})\in\mathbb{C}^{24}.
$$

For the above tensors, this gives
\[
V_L(\widetilde{\mathcal{A}}_1)=[-3,-2,7,7,2,5,-6,-4,-2,6,-4,2,
3,2,-7,-7,-2,-5,6,4,2,-6,4,-2]^T
\]
\[
V_L(\widetilde{\mathcal{A}}_2)=[3,-2,-6,1,-6,-2,9,-6,0,3,-10,8,
-3,2,6,-1,6,2,-9,6,0,-3,10,-8]^T
\]
\[
V_L(\widetilde{\mathcal{A}}_3)=[7,1,-12,5,3,-12,-6,-9,6,-2,-3,2,
-7,-1,12,-5,-3,12,6,9,-6,2,3,-2]^T
\]
\[
c=[12,1,-16,28,19,-24,-39,-29,14,3,-7,2,
-12,-1,16,-28,-19,24,39,29,-14,-3,7,-2]^T.
\]
A direct computation shows that $
\rho(M)=3,
\
\rho([M\mid c])=3.$ Thus, the corresponding linear system is consistent. Moreover,
since
$
\rho(M)=p=3,
$ the matrix $M$ has full column rank. Hence, by Corollary~\ref{cor3}, the system admits a unique solution.

Finally, applying the solution formula given in Remark~\ref{remark3.1} yields
$X=\begin{bmatrix}2&-1&3\end{bmatrix}^{T}.$
Therefore, $$
{X=[2,-1,3]^T}
$$ is the unique solution of the tensor equation
$\mathcal{A}\ltimes\mathcal{X}\ltimes\mathcal{B}=\mathcal{C}$.
\end{example}

\begin{example}
Consider the tensors
$\mathcal{A}\in\mathbb{C}^{3\times 6\times 2}$ and
$\mathcal{B}\in\mathbb{C}^{1\times 2\times 2}$ given in the previous example.
Now, choose
$$
\begin{tabular}{cccc|cccc}
    \hline
    \multicolumn{4}{c}{$\mathcal{C}(:,:,1)$} &
    \multicolumn{4}{c}{$\mathcal{C}(:,:,2)$} \\
    \hline
    13&-39&-12&39&
    28&3&-28&-3\\
    1&-29&-1&29&
    19&-7&-19&7\\
    -16&14&16&-14&
    -24&2&24&-2\\
    \hline
\end{tabular}.
$$
Following the construction preceding Remark~\ref{remark3.1}, the matrix M remains unchanged. For the above choice of $\mathcal C$, we have
\[
c
=
[13,1,-16,28,19,-24,-39,-29,14,3,-7,2,
-12,-1,16,-28,-19,24,39,29,-14,-3,7,-2]^T.
\]
Since the matrix $M$ is unchanged, so $ 
\rho(M)=3.$
However, for the modified vector $c$, direct computation gives
$
\rho([M\mid c])=4.
$
Therefore, $$
\rho(M)\neq \rho([M\mid c]).
$$
Hence, by Remark~\ref{remark3.1} the corresponding linear
system is inconsistent. Consequently, the tensor equation has no solution.
\end{example}

\begin{example}
Take the tensors $\mathcal{A,\ B,\ C}$ as follows
    $$\begin{tabular}{cccc|cccc}
    \hline
  \multicolumn{4}{c}{${{\mathcal{A}}}(:,:,1)$} & 
  \multicolumn{4}{c}{${\mathcal{A}}(:,:,2)$}  \\
\hline
    1&0&1&0&0&1&0&\\
    0&1&0&1&1&0&1&0
    \end{tabular} \ , \
    \begin{tabular}{cc|cc}
    \hline
  \multicolumn{2}{c}{${{\mathcal{B}}}(:,:,1)$} & 
  \multicolumn{2}{c}{${\mathcal{B}}(:,:,2)$}  \\
\hline
   1&-1&0&-1
    \end{tabular} 
   \  \text{and}\
    \begin{tabular}{cccc|cccc}
    \hline
  \multicolumn{4}{c}{${{\mathcal{C}}}(:,:,1)$} & 
  \multicolumn{4}{c}{${\mathcal{C}}(:,:,2)$}  \\
\hline
    2&0&-2&2&0&2&2&-2\\
     0&2&2&-2&2&0&-2&2
    \end{tabular}.$$
Then the solution $X \in \mathbb{C}^2,$ and \ 
\begin{center}
\begin{tabular}{c c|c c}
    \hline
  \multicolumn{2}{c}{${\widehat{{\mathcal{A}}}_{1}(:,:,1)}$} & 
  \multicolumn{2}{c}{${\widehat{{\mathcal{A}}}_{1}(:,:,2)}$ } \\
\hline
   1&0&0&1\\
   0&1&1&0
    \end{tabular}, 
    \begin{tabular}{c c|c c}
    \hline
  \multicolumn{2}{c}{$\widehat{{\mathcal{A}}}_{2}(:,:,1)$} & 
  \multicolumn{2}{c}{$\widehat{\mathcal{A}}_{2}(:,:,2)$}  \\
\hline
   1&0&0&1\\
   0&1&1&0
    \end{tabular}.
\end{center}
The linear dependence of  $\mathcal{V}_L(\widetilde{\mathcal{A}}_1),\ \mathcal{V}_L(\widetilde{\mathcal{A}}_2)$ and $\mathcal{V}_L({\mathcal{C}})$  implies that Eq.(\ref{main equation}) has infinitely many solutions, which are parameterized by $X=[1+t,1-t]^T$, where $t \in \mathbb{C}.$
\end{example}
%%%%%%%%%%%%%%%%%%%%%%%%%%%%%%%%%%%%%%%%%%%%%%%%%%%%%%%%%%%%%%%%%%%%%%%%%%%
\subsection{The case corresponds to  \texorpdfstring{$h \neq m$}{m = h}}
We now investigate the compatibility of Eq.(\ref{main equation}) when $h\neq m$. We present a lemma that provides necessary conditions on the tensor dimensions for the Eq.(\ref{main equation}) to admit a solution, along with constraints on the size of the solution vector $X \in \mathbb{C}^p$.
\begin{lemma} \label{m neq h vector case}
The existence of a solution vector of size $p$ for Eq.(\ref{main equation}) with $h\neq m$ imposes the following necessary dimensional conditions on $\mathcal{A}$, $\mathcal{B}$, $\mathcal{C}$ and $p$:\vspace{-0.5em}
\begin{enumerate}[label=(\roman*)]
\item $p = \frac{nhb}{mak}$;
\vspace{-0.7em}
\item $\frac{h}{m}$ and $\frac{k}{b}$ must belong to $\mathbb{N}$;
\vspace{-0.7em}
\item $\alpha = \operatorname{gcd}\{\frac{h}{m}, a\}$,
      $\gcd\{\frac{k}{b}, \alpha\}=1$
      and $\gcd\{\frac{h}{m}, \frac{k}{b}\}=1$.
\end{enumerate}
\end{lemma}
\begin{proof}
Let $X \in \mathbb{C}^{p}$ satisfy Eq.(\ref{main equation}).  Then by Eq.(\ref{main equation}) dimension consistency implies
$\frac{m l_2 p}{n} = h,
\quad
\frac{b l_2}{a} = k, \quad l_4=r.$
It follows that
 $\frac{k}{b}\in \mathbb{N}$  and $p = \frac{n h b}{m a k}.$ Define $\frac{l_2}{l_1/p} = \frac{n h}{m l_1} = \alpha,$ then $\alpha\in \mathbb{N}$ . This together with $l_1 \mid n$ implies that
$\frac{h}{m}\in \mathbb{N}$  and $\alpha \mid \frac{h}{m}$.
Because
\[
l_2 = \frac{a k}{b}
= \operatorname{lcm}\left\{\frac{l_1}{p}, a\right\},
\quad
\frac{l_2}{l_1/p} = \alpha,
\quad
\frac{l_2}{a} = \frac{k}{b},
\]
then $\gcd\left\{\alpha, \frac{k}{b}\right\} = 1$. Moreover, \vspace{-0.5em}
\[
l_1 = \operatorname{lcm}\{n,p\},
\qquad
\frac{l_1}{n} = \frac{h}{m\alpha},
\qquad
\frac{l_1}{p} = \frac{ak}{b\alpha},
\]
thus $\gcd\left\{\frac{h}{m\alpha}, \frac{ak}{b\alpha}\right\} = 1.$ We emphasize that $\gcd\left\{\alpha, \frac{k}{b}\right\} = 1$, then $\alpha \mid a$,
this together with $\alpha \mid (\frac{h}{m})$ gives
$\alpha = \gcd\{\frac{h}{m}, a\}
\  \text{and} \ 
\gcd\{\frac{h}{m}, \frac{k}{b}\} = 1.$
\end{proof}
Lemma~\ref{m neq h vector case} provides dimensional requirements that must be satisfied for Eq.(\ref{main equation}) to admit a solution. These constraints are known as \emph{admissibility conditions}; when they are fulfilled, the tensors $\mathcal{A}$, $\mathcal{B}$, and $\mathcal{C}$ are considered to be admissible. We next validate above result with various examples.
\begin{example} \label{example2} Consider the tensors $\mathcal{A}\in \mathbb{C}^{2 \times 3\times 2},$ $\mathcal{B}\in \mathbb{C}^{1 \times 2 \times 2}$ and $\mathcal{C}\in \mathbb{C}^{4 \times 6\times 2}$
\begin{center}
    \begin{tabular}{ccc|ccc}
    \hline
  \multicolumn{3}{c}{${{\mathcal{A}}}(:,:,1)$} & 
  \multicolumn{3}{c}{${\mathcal{A}}(:,:,2)$}  \\
\hline
 1 & -1&2&2&1&2\\
0 & 1&-1& -1&2&1
 \end{tabular}\ , \ 
 \begin{tabular}{cc|cc}
    \hline
  \multicolumn{2}{c}{${{\mathcal{B}}}(:,:,1)$} & 
  \multicolumn{2}{c}{${\mathcal{B}}(:,:,2)$}  \\
\hline
 3&2&0&1
 \end{tabular}
\end{center}
    \begin{center}
     \begin{tabular}{cccccc|cccccc}
    \hline
  \multicolumn{6}{c}{${{\mathcal{C}}}(:,:,1)$} & 
  \multicolumn{6}{c}{${\mathcal{C}}(:,:,2)$}  \\
\hline
 3&12&-3&4&12&-1&6&12&3&5&12&1\\
 -6&3&12&-2&4&12& 6&6&12&6&5&12\\
  0&-6&3&-1&-2&4&-3&6&6&-2&2&5\\
  6&0&-6&8&-1&-2& 12&-3&6&10&-2&2
 \end{tabular}.  
\end{center}
It is easy to verify that the compatibility conditions follow from Lemma~\ref{m neq h vector case}. By Straightforward calculation, we obtain $X=[1,2]^T$ solve Eq.(\ref{main equation}), which can be easily verified.

\end{example}
\begin{example}
Let $\mathcal{A}$ and  $\mathcal{B}$ as in Example \ref{example2}, and take $\mathcal{C}$ as follows
    \begin{center}
     \begin{tabular}{cccccc|cccccc}
    \hline
  \multicolumn{6}{c}{${{\mathcal{C}}}(:,:,1)$} & 
  \multicolumn{6}{c}{${\mathcal{C}}(:,:,2)$}  \\
\hline
 3&12&-3&4&12&-1&6&12&3&5&12&1\\
 -6&3&12&-2&4&12& 6&6&12&6&5&12\\
  1&-6&3&-1&-2&4&-3&6&6&-2&2&5\\
  6&1&-6&8&-1&-2& 12&-3&6&10&-2&2
 \end{tabular} . 
\end{center}
The absence of a solution despite compatibility of the tensors demonstrates that the compatibility conditions constitute necessary, rather than sufficient, criteria.
\end{example}
\begin{remark}
If $\frac{ml_1}{n} = \frac{ml_{2}p}{n}=h,\ \frac{l_1}{p}=\frac{bl_2}{a}=k$ and $\mathcal{B}$ is an identity tensor then $\mathcal{A \ltimes X \ltimes B= C}$ reduces to $\mathcal{A \ltimes X = C.}$
\end{remark}
Under the semi-tensor product (STP), the resulting tensor inherits a
block Toeplitz structure, which follows from the construction of the
STP of two tensors. We next characterize the structure of
$\mathcal{C}$ in Eq.~\ref{main equation} that is required for the
existence of a solution. In particular, we show that $\mathcal{C}$
possesses a block Toeplitz structure under the corresponding
dimensional compatibility conditions.

Toeplitz matrices arise naturally in various applications and have
well-known structural properties. Extending this concept to
third-order tensors, we call a tensor a Toeplitz tensor if each of its
frontal slices is a Toeplitz matrix. The following definition
formalizes this notion.
\begin{definition}\label{toeplitz def}
(Toeplitz Tensor \cite{J.FathitensorequationAX=BunderSTP}) A third-order tensor $\mathcal{A} = (a_{i_1,i_2,i_3}) \in \mathbb{C}^{n_1 \times n_2 \times n_3}$ is called a Toeplitz tensor if
\vspace{-0.8em}
\[
a_{i_1,i_2,i_3} = g_{\,i_1 - i_2,\, i_3},
\qquad
1 \le i_1 \le n_1,\;
1 \le i_2 \le n_2,\;
1 \le i_3 \le n_3,
\]
where $\mathcal{G} = (g_{k_1,k_2}) \in \mathbb{C}^{(n_1+n_2-1) \times n_3} \ \text{for}
\ 1-n_2 \le k_1 \le n_1-1  \ \text{and}\ 1 \le k_2 \le n_3.$ 
\end{definition}
\begin{theorem} \label{toeplitz theorem}
Suppose that the tensor vector Eq.~\eqref{main equation} is solvable. Then $\mathcal{C}$ admits a block partition in which every block is a Toeplitz tensor. More precisely, $\mathcal{C}$ admits the block
representation
\begin{equation}
\mathcal{C}=
\begin{bmatrix}
\operatorname{Block}_{1,1}(\mathcal{C}) & \operatorname{Block}_{1,2}(\mathcal{C}) & \cdots & \operatorname{Block}_{1,b}(\mathcal{C})\\
\operatorname{Block}_{2,1}(\mathcal{C}) & \operatorname{Block}_{2,2}(\mathcal{C}) & \cdots & \operatorname{Block}_{2,b}(\mathcal{C})\\
\vdots & \vdots & \ddots & \vdots\\
\operatorname{Block}_{m,1}(\mathcal{C}) & \operatorname{Block}_{m,2}(\mathcal{C}) & \cdots & \operatorname{Block}_{m,b}(\mathcal{C})
\end{bmatrix},
\end{equation} \nonumber
where each $\operatorname{Block}_{t,s}(\mathcal{C})$ ($t=1,\ldots,m$, $s=1,\ldots,b$) is a Toeplitz tensor of same dimension i.e. all blocks are in
$\mathbb{C}^{\frac{m}{h} \times \frac{k}{b} \times r}$.
\end{theorem}
\begin{proof}
Let $\operatorname{Row}_t(\mathcal{A})$ be the
$t^{th}$ horizontal slices of $\mathcal{A}$, and
$\operatorname{Col}_s(\mathcal{B})$ be the $s^{th}$ lateral slices of $\mathcal{B}$. According to proof of Lemma~\ref{m neq h vector case}, the following relations hold:
$\frac{h}{m}=\frac{l_1\alpha}{n},\ 
\frac{l_2}{l_1/p}=\alpha,\ \frac{l_2}{a}=\frac{k}{b}\ 
\text{and}\ 
 l_3=r.$
Now suppose that Eq.~(\ref{main equation}) admits a solution
$X=[x_1,x_2,\ldots,x_p]^T\in\mathbb{C}^p,$ under this assumption, we obtain the following:
\begin{align*}
\mathcal{A}\ltimes X\ltimes \mathcal{B}
&=(\mathcal{A}\otimes I_{\frac{h}{m \alpha}})\,(X\otimes \mathcal{I}_{\frac{ak}{b \alpha} \times \frac{ak}{b \alpha }\times l_3})\ltimes \mathcal{B}\\
&=
\begin{bmatrix}
(\operatorname{Row}_1(\mathcal{A})\otimes I_\frac{h}{m \alpha})(X\otimes \mathcal{I}_{\frac{ak}{b \alpha} \times \frac{ak}{b \alpha }\times l_3} )\\
(\operatorname{Row}_2(\mathcal{A})\otimes I_\frac{h}{m \alpha})(X\otimes \mathcal{I}_{\frac{ak}{b \alpha} \times \frac{ak}{b \alpha }\times l_3})\\
\vdots\\
(\operatorname{Row}_m(\mathcal{A})\otimes I_\frac{h}{m \alpha})(X\otimes \mathcal{I}_{\frac{ak}{b \alpha} \times \frac{ak}{b \alpha }\times l_3})
\end{bmatrix}
\ltimes \mathcal{B}\\
&=
\begin{bmatrix}
\operatorname{Row}_1(\mathcal{A})\ltimes X\\
\operatorname{Row}_2(\mathcal{A})\ltimes X\\
\vdots\\
\operatorname{Row}_m(\mathcal{A})\ltimes X
\end{bmatrix}
\otimes I_\alpha\;(\mathcal{B}\otimes I_\frac{k}{b})\\
&=
\begin{bmatrix}
(\operatorname{Row}_1(\mathcal{A})\ltimes X)\otimes I_\alpha\\
(\operatorname{Row}_2(\mathcal{A})\ltimes X)\otimes I_\alpha\\
\vdots\\
(\operatorname{Row}_m(\mathcal{A})\ltimes X)\otimes I_\alpha
\end{bmatrix}
\begin{bmatrix}
\operatorname{Col}_1(\mathcal{B})\otimes I_{k/b} &
\operatorname{Col}_2(\mathcal{B})\otimes I_{k/b} &
\cdots &
\operatorname{Col}_b(\mathcal{B})\otimes I_{k/b}
\end{bmatrix}
\end{align*}
\[
=
\begin{bmatrix}
\operatorname{Row}_1(\mathcal{A})\ltimes X\ltimes \operatorname{Col}_1(\mathcal{B}) &
\operatorname{Row}_1(\mathcal{A})\ltimes X\ltimes \operatorname{Col}_2(\mathcal{B}) &
\cdots &
\operatorname{Row}_1(\mathcal{A})\ltimes X\ltimes \operatorname{Col}_b(\mathcal{B})\\
\operatorname{Row}_2(\mathcal{A})\ltimes X\ltimes \operatorname{Col}_1(\mathcal{B}) &
\operatorname{Row}_2(\mathcal{A})\ltimes X\ltimes \operatorname{Col}_2(\mathcal{B}) &
\cdots &
\operatorname{Row}_2(\mathcal{A})\ltimes X\ltimes \operatorname{Col}_b(\mathcal{B})\\
\vdots & \vdots & \ddots & \vdots\\
\operatorname{Row}_m(\mathcal{A})\ltimes X\ltimes \operatorname{Col}_1(\mathcal{B}) &
\operatorname{Row}_m(\mathcal{A})\ltimes X\ltimes \operatorname{Col}_2(\mathcal{B}) &
\cdots &
\operatorname{Row}_m(\mathcal{A})\ltimes X\ltimes \operatorname{Col}_b(\mathcal{B})
\end{bmatrix}
\]
We next examine the internal organization of the individual sub-blocks.
\begin{align}\label{equation}
\operatorname{Row}_t(\mathcal{A}) \ltimes X \ltimes \operatorname{Col}_s(\mathcal{B})
&=
\Bigl[
(\operatorname{Row}_t(\mathcal{A}) \otimes I_{\frac{h}{m\alpha}})
(X \otimes \mathcal{I}_{{\frac{ak}{b\alpha}}\times \frac{ak}{b\alpha} \times r})
\otimes I_{\alpha}
\Bigr]
(\operatorname{Col}_s(\mathcal{B}) \otimes I_{\frac{k}{b}}) \nonumber\\
&=
(\operatorname{Row}_t(\mathcal{A}) \otimes I_{\frac{h}{m}})
(X \otimes \mathcal{I}_{{\frac{ak}{b}} \times {\frac{ak}{b}} \times r})
(\operatorname{Col}_s(\mathcal{B}) \otimes I_{\frac{k}{b}}) \nonumber\\
&=
(\operatorname{Row}_t(\mathcal{A}) \otimes I_{\frac{h}{m}})
\bigl[(X \otimes \mathcal{I}_{a \times a \times r})\operatorname{Col}_s(\mathcal{B})\bigr]
\otimes I_{\frac{k}{b}} .
\end{align}
Set the tensor $\mathcal{Y}=X\otimes \mathcal{B}\in\mathbb{C}^{pa\times b\times r},$ and decompose it into lateral blocks as $\mathcal{Y} =[\mathcal{Y}_1,\mathcal{Y}_2,\ldots,\mathcal{Y}_b],$
where
\begin{align*}
\mathcal{Y}_s
&=(X\otimes \mathcal{I}_{a \times a \times r})\operatorname{Col}_s(\mathcal{B})\\
&=[(x_1 b_{1,s,i}, x_1 b_{2,s,i},\,\ldots, x_1 b_{a,s,i}), \
    (x_2 b_{1,s,i}, \ldots, x_2b_{a,s,i}), \ldots (x_pb_{1,s,i},\ldots, x_p b_{a,s,i})]^{T}\\
&=[y_{1,s,i},  y_{2,s,i},\ldots, y_{pa,s,i}]^{T}
\in\mathbb{C}^{pa \times 1\times r}.
\end{align*}
Since $\frac{h}{m}$, $\frac{k}{b}$ has no common divisor other than $1$, the quantity $\frac{k}{b}$ can be written in the form
\[
\frac{k}{b} = m_1 \frac{h}{m} + m_2,
\]
for \(t=1,2,\ldots,m\) and \(s=1,2,\ldots,b\), \(i=k=1,2\cdots ,r,\) Eq.~(\ref{equation}) leads to
\begin{align*}
&\operatorname{Row}_t(\mathcal{A})\ltimes X\ltimes \operatorname{Col}_s(\mathcal{B})\\
&~=(\operatorname{Row}_t(\mathcal{A})\otimes I_\frac{h}{m})(\mathcal{Y}_s\otimes I_\frac{k}{b})\\
&~=y_{1,s,k}{^{P_k}\mathcal{A}^{1}_{:,:,i}}+y_{2,s,k}{^{P_k}\mathcal{A}^{2}_{:,:,i}}+\cdots+ y_{\frac{h}{m},s,k} 
{^{P_k}\mathcal{A}^{\frac{h}{m}}_{:,:,i}}+y_{\frac{h}{m}+1,s,k}{^{P_k}\mathcal{A}^{\frac{h}{m}+1}_{:,:,i}}+\cdots+y_{pa,s,k}{^{P_k}\mathcal{A}^{pa}_{:,:,i}}, \\
 & \hspace{6cm}  \text{ where } P_k \mathcal{A}^{1}_{:,:,i}, \ldots, P_k \mathcal{A}^{pa}_{:,:,i} \text{ are defined below}\\
&~=Block_{ts} (\mathcal{C})\in\mathbb{C}^{\frac{h}{m}\times \frac{k}{b} \times r}.
\end{align*}It follows that each $\operatorname{Block}_{ts}(\mathcal{C})$ has a Toeplitz structure. Consequently, the existence of a solution to Eq.~\eqref{main equation} requires that the tensor $\mathcal{C}$ must have block Toeplitz structure.\\
 
\begin{align*}
&{^{P_k}\mathcal{A}^{1}_{:,:,i}} := \left[
\begin{array}{*{14}c} 
   a_{t,1,i}& & & & & & a_{t,m_{1},i} & & & & & a_{t,(m_1+1),i}& &\\
   &\ddots &  & & & & &\ddots & & & & &\ddots & \\
   &&a_{t,1,i}& & & \cdots & & & a_{t,m_{1},i} & & & & & a_{t,(m_1+1)}\\
   & & &\ddots & & & & & &\ddots & & & &\\
   & & & & a_{t,1,i} & & & & & & a_{t,m_{1},i}& & &\\
\end{array}
\right],\\
\end{align*}
\begin{align*}
&{^{P_k}\mathcal{A}^{2}_{:,:,i}}=\\ &\small{\left[
\begin{array}{*{12}c}
& & &  a_{t,(m_1+2),i} & & & & &  &a_{t,\left[\frac{2k/b}{h/m}\right]+1,i}& &\\
& & & & \ddots & & & & && \ddots &\\
 a_{t,(m_1+1),i}& & & & &  a_{t,(m_1+2),i} & & &\cdots  & & & a_{t,\left[\frac{2k/b}{h/m}\right]+1,i}\\
 &\ddots& & & & &\ddots & & & & &\\
 & & a_{t,(m_1+1)i}& & & & &  a_{t,(m_1+2),i} & & & &\\
    \end{array}
\right],}\\
&\vdots\\
&{^{P_k}\mathcal{A}^{\frac{h}{m}}_{:,:,i}}:=\left[
\begin{array}{*{9}c} 
    & & & &a_{t,\frac{k}{b},i}& & & &\\
    & & & & &\ddots & & &\\
    a_{t,\frac{k}{b}-m_1,i}& & &\cdots& & & a_{t,\frac{k}{b},i}& &\\
    &\ddots & & & & & &\ddots &\\
   & & a_{t,\frac{k}{b}-m_1,i} & & & & & & a_{t,\frac{k}{b},i}
\end{array}
\right],\\
&{^{P_k}\mathcal{A}^{\frac{h}{m}+1}_{:,:,i}}:= \\&\small{ \left[
\begin{array}{*{14}c} 
 a_{t,\frac{k}{b}+1,i}& & & & && a_{t,\frac{k}{b}+m_{1},i}& & & &a_{t,\frac{k}{b}+m_{1}+1,i} & & & \\
 &\ddots& & & & &&\ddots& & &  & \ddots&& \\
 & &a_{t,\frac{k}{b}+1,i}& & &\cdots & &&a_{t,\frac{k}{b}+m_{1},i}& & &&& a_{t,\frac{k}{b}+m_{1}+1,i}\\
 & & &\ddots& & & & &&\ddots& & & & \\
 & & & &a_{t,\frac{k}{b}+1,i}& & & & & &a_{t,\frac{k}{b}+l_{1},i}& & &\\
\end{array}
\right],}\\
\vdots\\
&{^{P_k}\mathcal{A}^{pa}_{:,:,i}}:=\left[
\begin{array}{*{9}c} 
   & & & &a_{t,n,i}& & & &\\
    & & & & &\ddots & & &\\
    a_{t,n-m_1,i}& & &\cdots& & & a_{t,n,i}& &\\
    &\ddots & & & & & &\ddots &\\
   & & a_{t,n-m_1,i} & & & & & & a_{t,n,i}
\end{array}
\right],
\end{align*}
and for $r,k \in \mathbb{Z}$, 
$P_k = \{ r-k+2,\ldots,r,1,\ldots,r-k+1 \},\ 
\widetilde{(k:r)} =
\{ r-k+2,\ldots,r \},\  \text{when}\  k \le r,\ \text{and}\ 
\varnothing,\ \text{when}\  k>r.$
\end{proof}
Based on the definition of a Circulant matrix, we now introduce its Tensor analogue, referred to as a Circulant Tensor.
\begin{definition}\label{toeplitz def}
(Circulant Tensor) Denote $[n]=\{1,2,\ldots,n\}.$ Let $\mathcal{A} = (a_{i_1 \cdots i_m})$ be a real $m^{th}$ order $n$-dimensional tensor. If for $i_l, k_l \in [n]$ satisfying 
$k_l \equiv i_l + 1 \pmod{n}, \ l \in [m],$ we have
$$a_{i_1 \cdots i_{m-1} i_m} = a_{k_1 \cdots k_{m-1} i_m},$$
then we say that $\mathcal{A}$ is an $m^{th}$ order circulant tensor. $\mathcal{A} := \operatorname{circ}(\mathcal{A}),$
is referred to as Circulant tensor
whenever all of its frontal slices are Circulant matrices.
\end{definition}
\begin{definition}(F-diagonal Tensor \cite{M.Kilmer2011factorizationfortensor})
A third order tensor $\mathcal{D}\in \mathbb{C}^{n\times n\times  r}$ is F-diagonal, if its frontal slices are all diagonal. 
    
\end{definition}
\begin{corollary}
    Assume that the tensor–vector Eq.(\ref{main equation}) has a solution. If $\mathcal{A, B}$ are F-diagonal tensors with identical diagonal entries, then each sub-tensor of $\mathcal{C}$ exhibits a Circulant tensor structure.
\end{corollary}
%%%%%%%%%%%%%%%%%%%%%%%%%%%%%%%%%%%%%%%%%%%%%%%%%%%%%%%%%%%%%%%%%%%%%%%%%%%%%%%%%%%%%%%%%%%%%%
\section{Solvability analysis for tensor Eq. with matrix solutions}
Here we examine solvability properties of tensor matrix Eq. formulated through the STP:
\begin{equation} \label{main equation 2}
\underbrace{\underbrace{\mathcal{A} \ltimes X}_{\frac{m l_1}{n} \times \frac{l_1q}{p} \times l_3}\ltimes
\mathcal{B}}_{\frac{m l_2 p}{nq} \times \frac{b l_2}{a} \times l_4}  = \mathcal{C} \in \mathbb{C}^{h \times k \times r} ,
 \end{equation}
where tensors $\mathcal{A} \in \mathbb{C}^{m \times n \times r},\
\mathcal{B} \in \mathbb{C}^{a \times b \times r}, \
\mathcal{C} \in \mathbb{C}^{h \times k \times r}$ 
are known, while $X \in \mathbb{C}^{p \times q}$ is an unknown matrix to be determined. Here $l_1 = \operatorname{lcm}(n,p), \ l_3 = \operatorname{lcm}(r,1), \
l_2 = \operatorname{lcm}\!\left(\frac{q l_1}{p}, a\right),\ l_4=\operatorname{lcm}(l_3,r).$  As in the tensor–vector formulation, we begin by examining the special case $h=m$ and then proceed to the arbitrary dimension case.
\subsection{The special case corresponds to \texorpdfstring{$h=m$}{m = h}}
Here we investigate the compatibility of Eq.~(\ref{main equation 2}) when $h= m$. We present a lemma that provides necessary conditions on the tensor dimensions for the Eq.~(\ref{main equation 2}) to admit a solution, along with constraints on the size of the solution matrix $X\in \mathbb{C}^{p\times q}$.
\begin{lemma}\label{m=h matrix case}
Suppose that the tensor matrix Eq.~(\ref{main equation 2}) with $h=m$ admits a solution of size $p \times q$. Then the dimensions of the tensors  $\mathcal{A}$, $\mathcal{B}$,
$\mathcal{C}$, together with the integers $p$ and $q$, necessarily hold the resulting constraints:
\vspace{-0.5em}
\begin{enumerate}[label=(\roman*)]
\item $\frac{k}{b}$ must belong to $\mathbb{N}$;
\vspace{-0.7em}
\item $p = \frac{n}{\beta},\ q = \frac{a k}{b \beta},$ where $\beta$ is a common factor of $n$ and $\frac{a k}{b}$.
\end{enumerate}
\begin{proof}
Let $X \in \mathbb{C}^{p \times q}$ satisfy  Eq.~(\ref{main equation 2}). By Eq.(\ref{main equation 2}) we get $\frac{m l_2 p}{nq} = h,
\quad
\frac{b l_2}{a} = k, \quad l_{4}=r.$
From these relations, it follows that, $p = \frac{b nq}{a k},\ 
\frac{k}{b} = \frac{l_2}{a}$
and hence $\frac{k}{b}\in \mathbb{N}$. Define $\frac{ak}{qb}=\frac{n}{p}= \beta,$ then $\beta\in \mathbb{N}$. Moreover, the conditions $\beta\mid n$ and
$\beta\mid\frac{a k}{b}$ imply that $n$ and
$\frac{a k}{b}$ have a common factor $\beta$.
\end{proof}
\end{lemma}
\begin{theorem}\label{theorem 4.1}
Consider the tensor-matrix Eq.~\ref{main equation 2} for the case $h=m$, with given 
$\mathcal{A} \in \mathbb{C}^{m\times n\times r}$, $\mathcal{B} \in \mathbb{C}^{a\times b\times r}$,
$\mathcal{C} \in \mathbb{C}^{m\times k\times r}$, and unknown $X \in \mathbb{C}^{p\times q}$. Suppose the compatibility conditions of Lemma~\ref{m=h matrix case} hold. Then Eq.~\ref{main equation 2} with $m=h$ is equivalent to 
\begin{equation}\label{eq:4.1prime}
\operatorname{bcirc}(\mathcal{A})
\begin{bmatrix}
(X\otimes I_\beta)(\mathcal{B}_1\otimes I_{k/b})\\
(X\otimes I_\beta)(\mathcal{B}_2\otimes I_{k/b})\\
\vdots\\
(X\otimes I_\beta)(\mathcal{B}_r\otimes I_{k/b})
\end{bmatrix}
=
\begin{bmatrix}
\mathcal{C}_1\\ \mathcal{C}_2 \\ \vdots \\ \mathcal{C}_r
\end{bmatrix}.
\end{equation}
\end{theorem}
\begin{proof}
Since $h=m$, by the definition of the semi-tensor product,
\[
\mathcal{A}\ltimes X
=
\left(\mathcal{A}\otimes
\mathcal{I}_{\frac{l_1}{n}\times\frac{l_1}{n}\times1}\right)
*
\left(X\otimes
\mathcal{I}_{\frac{l_1}{p}\times\frac{k_1}{p}\times r}\right),
\]
\[
\mathcal{A}\ltimes X
=
\mathcal{A}*\mathcal{P},\  \text{where}\
\mathcal{P}:=
X\otimes
\mathcal{I}_{\beta\times\beta\times r}.
\]

Since the identity tensor
$\mathcal{I}_{\beta\times\beta\times r}$
has first frontal slice $I_\beta$ and all remaining frontal slices equal to zero, the frontal slices of $\mathcal{P}$ are
\[
\mathcal{P}_l=
\begin{cases}
X\otimes I_\beta,& l=1,\\
0,& l=2,\ldots,r.
\end{cases}
\]
Now
\[
(\mathcal{A}\ltimes X)\ltimes\mathcal{B}
=
(\mathcal{A}*\mathcal{P})\ltimes\mathcal{B}.
\]
Again, by the definition of the semi-tensor product,
\[
(\mathcal{A}*\mathcal{P})\ltimes\mathcal{B}
=
(\mathcal{A}*\mathcal{P})*\mathcal{Q},\ \text{where}\ 
\mathcal{Q}
:=
\mathcal{B}\otimes
\mathcal{I}_{\frac{k}{b}\times\frac{k}{b}\times1},
\]
whose frontal slices are $
Q_l=\mathcal{B}_l\otimes I_{k/b},
\,
l=1,\ldots,r.$
Using the associativity property of t-product,
\[
(\mathcal{A}*\mathcal{P})*\mathcal{Q}
=
\mathcal{A}*(\mathcal{P}*\mathcal{Q}).
\]

Now,
\[
\mathcal{P}*\mathcal{Q}
=
\operatorname{fold}
\left(
\operatorname{bcirc}(\mathcal{P})
\operatorname{unfold}(\mathcal{Q})
\right).
\]
Since only the first frontal slice of $\mathcal{P}$ is nonzero,
\[
\operatorname{bcirc}(\mathcal{P})
=
\operatorname{bdiag}(\mathcal{P}_1,\ldots,\mathcal{P}_1).
\]
Therefore,
\[
\operatorname{bcirc}(\mathcal{P})
\operatorname{unfold}(\mathcal{Q})
=
\begin{bmatrix}
\mathcal{P}_1\mathcal{Q}_1\\
\mathcal{P}_1\mathcal{Q}_2\\
\vdots\\
\mathcal{P}_1\mathcal{Q}_r
\end{bmatrix}.
\]
Since
$
\mathcal{P}_1=X\otimes I_\beta,
\
\mathcal{Q}_l=\mathcal{B}_l\otimes I_{k/b},$
we obtain
$
\mathcal{P}_1\mathcal{Q}_l
=
(X\otimes I_\beta)(\mathcal{B}_l\otimes I_{k/b}),
\, l=1,\ldots,r.$
Hence,  \begin{align*}
&\mathcal{A}\ltimes X \ltimes \mathcal{B}= \mathcal{A}*(\mathcal{P}*\mathcal{Q})=\mathcal{C}\\ 
\Rightarrow &\operatorname{fold} (\operatorname{bcirc}(\mathcal{A})
\operatorname{unfold}(\mathcal{P}*\mathcal{Q}))=\mathcal{C}\\
\Rightarrow & \operatorname{fold} \left ( \operatorname{bcirc}(\mathcal{A})
\begin{bmatrix}
(X\otimes I_\beta)(\mathcal{B}_1\otimes I_{k/b})\\
(X\otimes I_\beta)(\mathcal{B}_2\otimes I_{k/b})\\
\vdots\\
(X\otimes I_\beta)(\mathcal{B}_r\otimes I_{k/b}) 
\end{bmatrix}\right )=\mathcal{C}\\
 \Rightarrow & \operatorname{bcirc}(\mathcal{A})
\begin{bmatrix}
(X\otimes I_\beta)(\mathcal{B}_1\otimes I_{k/b})\\
(X\otimes I_\beta)(\mathcal{B}_2\otimes I_{k/b})\\
\vdots\\
(X\otimes I_\beta)(\mathcal{B}_r\otimes I_{k/b}) 
\end{bmatrix}=\begin{bmatrix}
    \mathcal{C}_{1}\\
    \mathcal{C}_{2}\\
    \vdots\\
    \mathcal{C}_{r}
\end{bmatrix}.
\end{align*}
\end{proof}
\begin{remark}
Suppose that $\frac{a k}{b}$ and $n$ have a common factor $\beta_j$ for $j = 1, \dots, s$. Then, Eq.~(\ref{main equation 2}) can admit solutions of sizes $p_j \times q_j$, here $p_j = \frac{n}{\beta_j}, \ q_j = \frac{a k}{b \beta_j}.$ We refer to these as permissible sizes, and note the following properties.\vspace{-0.5em}
\begin{enumerate}[label=(\roman*)]
\item Suppose the equation admits solutions of two unequal allowable dimensions,
represented by $X^{p_1 \times q_1}$ and $X^{p_2 \times q_2}$ and satisfy $\frac{q_2}{q_1} = \frac{p_2}{p_1} \in \mathbb{Z} > 1.$
Then the larger solution can be expressed as, $X^{p_2 \times q_2} = X^{p_1 \times q_1} \otimes I_{\frac{q_2}{q_1}}.$ \vspace{-0.7em}
\item Define $\tilde{\beta} = \gcd\{n, \frac{a k}{b}\}, \ \tilde{p} = \frac{n}{\tilde{\beta}}, \ \tilde{q} = \frac{a k}{b \tilde{\beta}}.$ If Eq.~(\ref{main equation 2}) admits solution with minimal size $\tilde{p} \times \tilde{q}$ , then a solution exists for every other permissible size as well.\vspace{-0.7em}
\item In the special case where $b = k$ and  $\beta = 1$, then Eq.~(\ref{main equation 2}) to reduces to the corresponding tensor matrix equation under the standard t-product.
\end{enumerate}
\end{remark}
%%%%%%%%%%%%%%%%%%%%%%%%%%%%%%%%%%%%%%%%%%%%%%%%%%%%
\begin{algorithm}
\caption{Matrix Solution of $\mathcal{A}\ltimes X\ltimes\mathcal{B}=\mathcal{C}$ (case $m=h$)}
\label{algorithm2}
\begin{algorithmic}[1]
\State \textbf{Input} $\mathcal{A}\in\mathbb{C}^{m\times n\times r}$, $\mathcal{B}\in\mathbb{C}^{a\times b\times r}$, $\mathcal{C}\in\mathbb{C}^{m\times k\times r}$
\State Check admissibility (Lemma~\ref{m=h matrix case}):  $\dfrac{k}{b}\in\mathbb{N}$ and find a common factor $\beta$ of $n$ and $\dfrac{ak}{b}$; take $\beta \leftarrow \gcd\!\left(n,\dfrac{ak}{b}\right)$ for the minimal admissible size.
\If{no such $\beta$ exists}
    \State \textbf{Output:} ``Eq.~(\ref{main equation 2}) is inadmissible - no solution.''
\EndIf
\State Compute $p \leftarrow \dfrac{n}{\beta}$, $q \leftarrow \dfrac{ak}{b\beta}$.
\For{$i=1,\ldots,p$ and $j=1,\ldots,q$}\\
    Apply $\mathcal{L}$ (Theorem~\ref{theorem 4.1}) to the basis matrix $E_{ij}\in\mathbb{C}^{p\times q}$:
    $$
    \mathcal{L}(E_{ij}) = \mathrm{bcirc}(\mathcal{A})
    \begin{bmatrix}
    (E_{ij}\otimes I_\beta)(\mathcal{B}_1\otimes I_{k/b})\\
    \vdots\\
    (E_{ij}\otimes I_\beta)(\mathcal{B}_r\otimes I_{k/b})
    \end{bmatrix}
    $$
    \State Store $\mathcal{L}(E_{ij})$ as the corresponding column of $N$.
\EndFor
\State Compute the Moore-Penrose inverse $N^{\dagger}$.
\State Compute residual $\mathrm{res} \leftarrow NN^{\dagger}\,\mathrm{vec}(\mathcal{C}) - \mathrm{vec}(\mathcal{C})$.
\If{$\|\mathrm{res}\| > 0$}
    \State \textbf{Output:} ``Eq.~(\ref{main equation 2}) has no solution at size $p\times q$.''
\Else
    \State Compute $\mathrm{vec}(X) \leftarrow N^{\dagger}\,\mathrm{vec}(\mathcal{C}) + \big(I_{pq} - N^{\dagger}N\big)z,\ z\in\mathbb{C}^{pq}$ arbitrary.
    \State Reshape $\mathrm{vec}(X)$ into $X\in\mathbb{C}^{p\times q}$.
    \If{$\mathrm{rank}(N) = pq$}
        \State \textbf{Output:} unique solution $X$.
    \Else
        \State \textbf{Output:} general solution $X$ with free parameter $z\in\mathbb{C}^{pq}$.
    \EndIf
\EndIf
\end{algorithmic}
\end{algorithm}

%%%%%%%%%%%

\begin{example}
Take tensors $\mathcal{A} \in \mathbb{C}^{2 \times4 \times 2}$, \ $\mathcal{B} \in \mathbb{C}^{2 \times 1 \times 2}$ and $\mathcal{C} \in \mathbb{C}^{2 \times 4 \times 2}$ as follows; \\
 \[ \begin{tabular}{c c c c|c c c c}
\hline
  \multicolumn{4}{c}{$\mathcal{A}(:,:,1)$} & 
  \multicolumn{4}{c}{$\mathcal{A}(:,:,2)$} \\
  \hline
1&-1&2&0&0&1&-1&3 \\
0&1&-1&2&2&0&1&-1 \\
\hline
\end{tabular}, \quad
\begin{tabular}{c|c}
    \hline
  \multicolumn{1}{c}{$\mathcal{B}(:,:,1)$} & 
  \multicolumn{1}{c}{$\mathcal{B}(:,:,2)$} \\
  \hline 
1&-1\\
0&1\\
\hline
\end{tabular},\quad
    \begin{tabular}{c c c c| c c c c}
\hline
  \multicolumn{4}{c}{$\mathcal{C}(:,:,1)$} & 
  \multicolumn{4}{c}{$\mathcal{C}(:,:,2)$} \\
  \hline
2&-4&5&-3&-1&3&-4&6 \\
-2&2&-3&5&4&-1&3&-4 \\
\hline
\end{tabular}.\]
According to Lemmma~(\ref{m=h matrix case}), the permissible solutions sizes are $1 \times 2, \ 2 \times 4,\ 4 \times 8 $.Straightforward calculation shows that $X_1=[2,\ 1]$ is a solution of tensor matrix Eq.(\ref{main equation 2}). Moreover, the matrices $$X_2= X_1 \otimes I_2, \quad X_3 = X_1 \otimes I_4$$ satisfy the Eq(\ref{main equation 2}) as well,  thereby providing further admissible solutions. 
\end{example}
%%%%%%%%%%%%%%%%%%%%%%%%%%%%%%%%%%%%%%%%%%%%%%%%%%%%%%%%%%%%%%%%%%%%%%%%%%%%%%%%%%%%%%%%%%%%%%%%%%%%%%%%%%%%%%%%%%%%%%%%%
\subsection{The case corresponds to \texorpdfstring{$h\neq m$}{m = h}} 
We now investigate the compatibility of Eq.~(\ref{main equation 2}) when $h\neq m$. We present a lemma that provides necessary conditions on the tensor dimensions for the Eq.~(\ref{main equation 2}) to admit a solution, along with constraints on the size of the solution matrix $X\in \mathbb{C}^{p\times q}$.
\begin{lemma} \label{m neq h matrix case}
Assume that the tensor–matrix Eq.~(\ref{main equation 2}) with $h\neq m$ admits a solution
$X \in \mathbb{C}^{p \times q}$. Then the tensor dimensions and the integers $p,q$ necessarily hold the resulting constraints:
\vspace{-0.5em}
\begin{enumerate}[label=(\roman*)]
\item[(i)]  $\frac{h}{m}$ and $\frac{k}{b}$ must belong in $\mathbb{N}$;
\vspace{-0.7em}
\item[(ii)] $p=\frac{nh}{m\alpha}$, $q=\frac{ak}{b\alpha}$, here $\alpha$
divide both $\frac{nh}{m}$ and $\frac{ak}{b}$;
\vspace{-0.7em}
\item[(iii)]  If $\beta=\gcd\left\{\frac{h}{m},\alpha\right\}$ then $\gcd\left\{\frac{h}{m\beta},\frac{\alpha}{\beta}\right\}=1$,
$\gcd\{\beta,\frac{k}{b}\}=1$, $\beta \mid a$,
\end{enumerate}
\end{lemma}
\begin{proof}
By Eq.~\ref{main equation 2}, we have $\frac{m l_2 p}{n q} = h,\
\frac{b l_2}{a} = k, \ l_4=r.$
Because $\frac{l_2}{a}=\frac{k}{b}$, then $\frac{k}{b}\in \mathbb{N}$.
Moreover, \ $\frac{l_2}{q l_1/p}= \frac{n h}{m l_1}=\frac{h}{m}\cdot\frac{n}{l_1}= \beta.$ Since $l_1 \mid n$, the above quantity is an integer, which forces
$\frac{h}{m}\in \mathbb{N}$. Thus, condition (i) holds.\\
For (ii), because $\frac{l_2}{q l_1/p}\in\mathbb{N}$ and $p \mid l_1$, we conclude that
$q \mid l_2$. Combining this with $\frac{m l_2 p}{n q}=h, \  \frac{b l_2}{a}=k,$ yields
\[
\frac{l_2}{q}=\frac{nh}{mp}=\frac{ak}{bq}.
\]
Let $\alpha = \frac{l_2}{q}.$
Then $\alpha$ is a common factor of $\frac{nh}{m}$ and $\frac{ak}{b}$, and
\[
p=\frac{nh}{m\alpha}, \quad q=\frac{ak}{b\alpha}.
\] 
For (iii), define $\beta := \frac{l_2}{q l_1/p}.$
From the above relations, $\beta \mid \frac{h}{m}$. Since
\[
l_1=\operatorname{lcm}(n,p), \quad
\frac{l_1}{n}=\frac{h}{m\beta}, \quad
\frac{l_1}{p}=\frac{\alpha}{\beta},
\]
we obtain
$\gcd\left\{\frac{h}{m\beta},\frac{\alpha}{\beta}\right\}=1,$
which implies $\beta=\gcd\{\frac{h}{m},\alpha\}.$
Furthermore, 
\[
l_2=\operatorname{lcm}\left\{a,\frac{q l_1}{p}\right\}, \quad
\frac{l_2}{a}=\frac{k}{b},
\]
we conclude that $\gcd\left\{\beta,\frac{k}{b}\right\}=1$. Finally, since
$\frac{l_1}{p}=\frac{ak}{bq\beta}$ belong to $\mathbb{N}$, implies that $\beta$ divides $a$.
\end{proof}
\begin{remark}
The solution dimensions characterized in Lemma~\ref{m neq h matrix case} are referred to as
\emph{admissible sizes}. These sizes possess the following structural properties:\vspace{-0.5em}
\begin{enumerate}[label=(\roman*)]
\item
When $\alpha = 1$, the corresponding dimensions are $p=\frac{nh}{m}$ and $q=\frac{ak}{b}$. In this case,
Eq.~(\ref{main equation 2}) simplifies to $(\mathcal{A}\otimes I_{\frac{h}{m}})\,(X\otimes \mathcal{I}_{1\times 1 \times r})\,(\mathcal{B}\otimes I_{\frac{k}{b}})
=\mathcal{C},$ which coincides with the classical formulation based on the standard t-product.
\vspace{-0.6em}
\item
Assume that $X$ is a solution of size $p_1\times q_1$ that satisfies the allowable size conditions. If there exists an integer $\ell>1$ such that $p_2=\ell p_1 \quad \text{and} \quad q_2=\ell q_1,$ then the enlarged matrix $X\otimes I_\ell$, having dimensions
$p_2\times q_2$, also solves Eq.~(\ref{main equation 2}).
\vspace{-0.6em}
\item
Let $\hat{\alpha}=\gcd\!\left(\frac{nh}{m},\,\frac{ak}{b}\right),$ and define the smallest allowable dimensions by
$p=\frac{nh}{m\hat{\alpha}},\  q=\frac{ak}{b\hat{\alpha}}.$ If Eq.~(\ref{main equation 2}) admits a solution with this minimal size, then solutions exist for every other allowable size determined by Lemma~(\ref{m neq h matrix case}).
\end{enumerate}
\end{remark}
\begin{example}\label{example4}
 Consider tensors $\mathcal{A} \in \mathbb{C}^{2 \times 2 \times 2},$ \ $\mathcal{B}\in \mathbb{C}^{2 \times 1 \times 2}$ and $\mathcal{C} \in \mathbb{C}^{4 \times 3 \times 2}$ as follows:
$$\begin{tabular}{c c| c c}
\hline
  \multicolumn{2}{c}{$\mathcal{A}(:,:,1)$} & 
  \multicolumn{2}{c}{$\mathcal{A}(:,:,2)$} \\
  
\hline
1&2&-1&2 \\
0&1&0&-1 \\
\end{tabular},\quad 
\begin{tabular}{c|c}
\hline
  \multicolumn{1}{c}{$\mathcal{B}(:,:,1)$} & 
  \multicolumn{1}{c}{$\mathcal{B}(:,:,2)$} \\
  \hline
1&1 \\
0&1\\
\end{tabular}
 \text{and} \
\begin{tabular}{c c c|c c c}
\hline
  \multicolumn{3}{c}{$\mathcal{C}(:,:,1)$} & 
  \multicolumn{3}{c}{$\mathcal{C}(:,:,2)$} \\
  
\hline
12&-6&3&12&6&3 \\
5&12&-6&5&12&6 \\
0&0&0&0&0&0\\
-2&0&0&2&0&0\\

\end{tabular}.$$ By Lemma~\ref{m neq h matrix case}, the admissible solution sizes are $2 \times 3 \ \text{and}\ 4 \times 6.$ One can directly verify that
 \[X_1=\begin{bmatrix}
    1&4&6\\
    3&2&0\\ 
\end{bmatrix} \in \mathbb{C}^{2 \times 3}\] satisfies the tensor–matrix Eq.(\ref{main equation 2}). Moreover, by expanding $X_1$ through a Kronecker product with the identity, the matrix $X_2 = X_1 \otimes I_2$ also fulfills Eq.(\ref{main equation 2}), yielding a solution of dimension $4 \times 6$.
\end{example}
It should be noted that the preceding lemma establishes only a necessary condition for the existence of a solution. The following example demonstrates that this condition is not sufficient in general.
\begin{example}
Consider $\mathcal{A}$ and $\mathcal{B}$ as in Example(\ref{example4}) and $\mathcal{C}$ as follows:
 $$\begin{tabular}{c c c|c c c}
\hline
  \multicolumn{3}{c}{$\mathcal{C}(:,:,1)$} & 
  \multicolumn{3}{c}{$\mathcal{C}(:,:,2)$} \\
\hline
12&-6&3&12&6&3 \\
5&12&-6&5&12&6 \\
1&0&0&1&0&1\\
-2&0&0&2&0&1\\
\end{tabular}.$$  \\
Though the given tensors are compatible, it has no solution.
\end{example}
\begin{remark}
If $\frac{ml_1}{n} = \frac{ml_{2}p}{nq}=h,\ \frac{l_1q}{p}=\frac{bl_2}{a}=k$ and $\mathcal{B}$ is an identity tensor then $\mathcal{A \ltimes X \ltimes B= C}$ reduces to $\mathcal{A \ltimes X = C.}$
\end{remark}
Analogous to Theorem \ref{toeplitz theorem}, we characterize the explicit structure of $\mathcal{C}$ using the properties of the semi-tensor product and the dimensional compatibility condition.
\begin{theorem}
Suppose that the tensor matrix Eq.~\eqref{main equation 2} is solvable.  Then $\mathcal{C}$ admits a block partition in which every block is a Toeplitz tensor. More precisely, $\mathcal{C}$ admits the block
representation

$$
\mathcal{C} =
\left[
\begin{array}{cccc}
\operatorname{Block}_{11}(\mathcal{C}) &
\operatorname{Block}_{12}(\mathcal{C}) &
\cdots &
\operatorname{Block}_{1b}(\mathcal{C}) \\
\vdots & \vdots & \vdots & \vdots \\
\operatorname{Block}_{\frac{h}{m\beta},1}(\mathcal{C}) &
\operatorname{Block}_{\frac{h}{m\beta},2}(\mathcal{C}) &
\cdots &
\operatorname{Block}_{\frac{h}{m\beta},b}(\mathcal{C}) \\
 \hdashline \\
\vdots & \vdots & \vdots & \vdots \\[6pt]
\hdashline \\
\operatorname{Block}_{\frac{h}{\beta}-\frac{h}{m\beta}+1,1}(\mathcal{C}) &
\operatorname{Block}_{\frac{h}{\beta}-\frac{h}{m\beta}+1,2}(\mathcal{C}) &
\cdots &
\operatorname{Block}_{\frac{h}{\beta}-\frac{h}{m\beta}+1,b}(\mathcal{C}) \\[4pt]
\vdots & \vdots & \vdots & \vdots \\
\operatorname{Block}_{\frac{h}{\beta},1}(\mathcal{C}) &
\operatorname{Block}_{\frac{h}{\beta},2}(\mathcal{C}) &
\cdots &
\operatorname{Block}_{\frac{h}{\beta},b}(\mathcal{C})
\end{array}
\right],$$\\
here each $\operatorname{Block}_{j,k}(\mathcal{C})$ is a Toeplitz tensor of identical size of $\frac{bh}{\beta}$ for $j=1, 2, \cdots ,\frac{h}{\beta}; k=1,2,\cdots ,b$.
In particular, all blocks belong to
$\mathbb{C}^{\beta \times \frac{k}{b} \times r}$.
\begin{proof}
By applying arguments analogous to those in the proof of Theorem~\ref{toeplitz theorem}, we conclude that
\[
\mathcal{A} \ltimes X
= (\mathcal{A} \otimes I_{\frac{h}{m\beta}})
(X \otimes \mathcal{I}_{\frac{\alpha}{\beta}\times \frac{\alpha}{\beta}\times r})
=
\left[
\begin{array}{cccc}
\operatorname{Block}_{11}(\mathcal{Z}) & \cdots & \operatorname{Block}_{1q}(\mathcal{Z}) \\
\operatorname{Block}_{21}(\mathcal{Z}) & \cdots & \operatorname{Block}_{2q}(\mathcal{Z}) \\
\vdots & \ddots & \vdots \\
\operatorname{Block}_{m1}(\mathcal{Z}) & \cdots & \operatorname{Block}_{mq}(\mathcal{Z})
\end{array}
\right]
= \mathcal{Z},
\]
where $\operatorname{Block}_{st}(\mathcal{Z}) \in
\mathbb{C}^{\frac{h}{m\beta} \times \frac{\alpha}{\beta} \times r}$
is a Toeplitz tensor for
$s = 1,\ldots,m$ and $t = 1,\ldots,q$.
Taking the $j^{th}$ horizontal slice of $\mathcal{Z}$ together with the $k^{th}$ lateral slice of
$\mathcal{B}$ yields
\[
\operatorname{Row}(\mathcal{Z})_j \ltimes \operatorname{Col}(\mathcal{B})_k
=
(\operatorname{Row}(\mathcal{Z})_j \otimes I_{\beta})
(\operatorname{Col}(\mathcal{B})_k \otimes I_{\frac{k}{b}})
=
\operatorname{Block}_{jk}(\mathcal{C}).
\]
Since $\mathcal{Z}$ is Toeplitz, each $\operatorname{Block}_{jk}(\mathcal{C})$ inherits the
Toeplitz structure, which completes the proof.
\end{proof}
\end{theorem}
\begin{remark}
     Assume that the tensor–matrix Eq.(\ref{main equation 2}) has a solution. If $\mathcal{A, B}$ are  F-diagonal tensors with identical diagonal entries, then each sub-tensor of $\mathcal{C}$ exhibits a Circulant tensor structure.
\end{remark}
%%%%%%%%%%%%%%%%%%%%%%%%%%%%%%%%%%%%%%%%%%%%%%%%%%%%%%%%%%%%%%%%%%%%%%%%%%%%%%%%%%%%%%%%%%%%%%%%%%%%%%%%%%%%%%%%%%%%
\section{\textbf{Solvability analysis for tensor  Eq. with tensor solutions}}
Here we examine solvability properties of tensor–tensor Eq. formulated through the STP:
\begin{align} \label{main equation 3}
 \underbrace{\underbrace{\mathcal{A} \ltimes \mathcal{X}}_{\frac{m l_1}{n} \times \frac{l_1q}{p} \times l_3}\ltimes
\mathcal{B}}_{\frac{m l_2 p}{nq} \times \frac{b l_2}{a} \times l_4}  = \mathcal{C} \in \mathbb{C}^{h \times k \times t} ,
 \end{align}
\noindent here $l_1 = \operatorname{lcm}\{p,n\}, \ l_3 = \operatorname{lcm}\{l,r\}, \
l_2 = \operatorname{lcm}\{a,\frac{q l_1}{p}\},\ l_4=\operatorname{lcm}\{s,l_3\}$ tensors $\mathcal{A} \in \mathbb{C}^{m \times n \times r},\
\mathcal{B} \in \mathbb{C}^{a \times b \times s}, \
\mathcal{C} \in \mathbb{C}^{h \times k \times t}$
are known, while $\mathcal{X} \in \mathbb{C}^{p \times q \times l}$ is a tensor to be determined. Following the approach adopted for the tensor–vector case, we begin by analyzing the special situation $h=m$ and subsequently address the general case $h\neq m$.

\subsection{The special case corresponds to \texorpdfstring{\(m=h\)}{m=h}}
Here we investigate the compatibility of Eq.(\ref{main equation 3}) when $h= m$. We present a lemma that provides necessary conditions on the tensor dimensions for the Eq.(\ref{main equation 3}) to admit a solution, along with constraints on the size of the solution tensor $\mathcal{X}\in \mathbb{C}^{p\times q \times l}$.
\begin{lemma}\label{m=h tensor case}
Suppose that the tensor tensor Eq.(\ref{main equation 3}) in the case $h=m$ admits a solution of dimension $p \times q \times l$. Then the sizes of the tensors  $\mathcal{A}$, $\mathcal{B}$,
$\mathcal{C}$, together with the integers $p$ , $q$ and $l$, necessarily hold the resulting constraints:\vspace{-0.5em}
\begin{enumerate}[label=(\roman*)]
\item $\frac{k}{b}$ must belong to $\mathbb{N}$;\vspace{-0.7em}
\item $p = \frac{n}{\beta},\ q = \frac{a k}{b \beta},$ where $\beta$ is a common factor of $n$ and $\frac{a k}{b}$.
\end{enumerate}
Moreover, $l_3= \operatorname{lcm(r,l)},\ l_4=\operatorname{lcm}(l_3,s)=t$
\begin{proof}
The proof follows from the matrix solution.
\end{proof}
\end{lemma}
\begin{remark}
Suppose $\frac{a k}{b}$ and $n$ have common factor $\beta_j$, $j = 1, \dots, s$. Then, Eq.(\ref{main equation 3}) may admit solutions whose dimensions are $p_j \times q_j \times g_j$, where $p_j = \frac{n}{\beta_j}, \ q_j = \frac{a k}{b \beta_j}.$ Such dimensions are referred to as admissible sizes. The following properties hold for solutions corresponding to different admissible sizes.\vspace{-0.5em}
\begin{enumerate}[label=(\roman*)]
\item Suppose that $\mathcal{X}^{p_1 \times q_1 \times g_1}$ and $\mathcal{X}^{p_2 \times q_2 \times g_2}$ are two solutions with distinct admissible sizes satisfying, $l=\frac{q_2}{q_1} = \frac{p_2}{p_1} \in \mathbb{Z}{>1} \ \text{and}\ \frac{l_2}{l_1}= w \in \mathbb{Z}{>1}.$
Then the solution can be constructed as, $\mathcal{X}^{p_2 \times q_2 \times g_2} = \mathcal{X}^{p_1 \times q_1 \times g_1} \otimes \mathcal{I}_{l \times l \times w}.$ \vspace{-0.7em}
\item Define $\tilde{\beta} = \gcd\!\left(n, \frac{a k}{b}\right), \ \tilde{p} = \frac{n}{\tilde{\beta}}, \ \tilde{q} = \frac{a k}{b \tilde{\beta}}.$ If Eq.(\ref{main equation 3}) admits a solution with the minimal size $\tilde{p} \times \tilde{q} \times \tilde{g}$, then a solution exists for every permissible size.

\end{enumerate}
\end{remark}
\begin{example}
Take tensor $\mathcal{A} \in \mathbb{C}^{2 \times 4 \times 2}, $ \ $\mathcal{B} \in \mathbb{C}^{2 \times 1 \times 2}$ \ and\ $\mathcal{C} \in \mathbb{C}^{2 \times 4 \times 2}$ as follows:
$$ \begin{tabular}{c c c c|c c c c}
\hline
  \multicolumn{4}{c}{$\mathcal{A}(:,:,1)$} & 
  \multicolumn{4}{c}{$\mathcal{A}(:,:,2)$} \\
  
\hline
-1&0&1&0&0&2&1&-1 \\
0&-1&0&1&1&2&1&0 \\

\end{tabular}, \quad
\begin{tabular}{c|c}
    \hline
  \multicolumn{1}{c}{$\mathcal{B}(:,:,1)$} & 
  \multicolumn{1}{c}{$\mathcal{B}(:,:,2)$} \\
  
\hline 
2&1\\
1&0\\
\end{tabular}\\ \ \text{and} \
    \begin{tabular}{c c| c c}
\hline
  \multicolumn{2}{c}{$\mathcal{C}(:,:,1)$} & 
  \multicolumn{2}{c}{$\mathcal{C}(:,:,2)$} \\
  
\hline
9&0&4&3 \\
10&9&7&4 \\

\end{tabular}.$$
The admissible solution dimensions can be determined from the preceding conditions as follows $2 \times 2 \times 1,\ 2 \times2 \times 2,\ 4 \times 4 \times 1 \ \text{and} \ 4 \times 4 \times 2 $ then tensor tensor Eq.(\ref{main equation 3}) solution is
$$\begin{tabular}{c c |c c}
\hline
  \multicolumn{2}{c}{$\mathcal{X}_1(:,:,1)$} & 
  \multicolumn{2}{c}{$\mathcal{X}_1(:,:,2)$} \\
  
\hline
1&2&2&0 \\
0&1&2&1 \\

\end{tabular}{\in \mathbb{C}^{2 \times 2 \times 2}}. $$ 
\end{example}
%%%%%%%%%%%%%%%%%%%%%%%%%%%%%%%%%%%%%%%%%%%%%%%%%%%%%%%%%%%%%%%%%%%%%%%%%%%%%%%%%%%%%%%%%%%%%%%%%%%%%%
\subsection{The case corresponds to \texorpdfstring{$h\neq m$}{m ≠ h}}
Here we investigate the compatibility of Eq.(\ref{main equation 3}) when $h\neq m$, with given tensors
$\mathcal{A} \in \mathbb{C}^{m \times n \times r}$,
$\mathcal{B} \in \mathbb{C}^{a \times b \times s}$,
$\mathcal{C} \in \mathbb{C}^{h \times k \times t}$. We first present a lemma that provides necessary conditions on the tensor dimensions for the Eq.(\ref{main equation 3}) to admit a solution, along with constraints on the size of the solution tensor $\mathcal{X}\in \mathbb{C}^{p\times q \times l}$.
\begin{lemma} \label{m neq h tensor case}
Assume that the tensor–tensor Eq.(\ref{main equation 3}) with $h\neq m$ admits a solution
$\mathcal{X} \in \mathbb{C}^{p \times q\times l}$. Then the tensor dimensions and the integers $p,q, l$ necessarily hold the resulting constraints:\vspace{-0.5em}
\begin{enumerate}[label=(\roman*)]
\item $\frac{h}{m}$ and $\frac{k}{b}$ belong to $\mathbb{N}$;\vspace{-0.7em}
\item $p=\frac{nh}{m\alpha}$, $q=\frac{ak}{b\alpha}$, here $\alpha$
divide both $\frac{nh}{m}$ and $\frac{ak}{b}$;\vspace{-0.7em}
\item If $\beta=\gcd\left\{\frac{h}{m},\alpha\right\}$ then  $\gcd\left\{\frac{nh}{m\beta},\frac{\alpha}{\beta}\right\}=1$,
$\gcd\left\{\beta,\frac{k}{b}\right\}=1$, $\beta \mid a$.\\
Moreover, $l_3= \operatorname{lcm}\left\{r,l\right\}, \ l_4=\operatorname{lcm}\left\{l_3,s\right\}=t.$
\end{enumerate}
\end{lemma}
\begin{proof}
The proof follows from the matrix solution.
\end{proof}
\begin{remark}
The solution dimensions characterized in Lemma~\ref{m neq h tensor case} are referred to as
\emph{admissible sizes}. These sizes possess the following structural properties.\vspace{-0.5em}

\begin{enumerate}[label=(\roman*)]
\item
Let $\mathcal{X}$ satisfy Eq.(\ref{main equation 3}) with admissible size $p_1 \times q_1 \times r_1$. Assume that 
$\frac{q_2}{q_1} = \frac{p_2}{p_1} \in \mathbb{Z}>1 \ \text{and} \ \frac{r_2}{r_1}=w \in \mathbb{Z}>1$ then the tensor $\mathcal{X} \otimes I_{\frac{p_2}{p_1} \times \frac{p_2}{p_1} \times wr_1}$ also satisfy the Eq.(\ref{main equation 3}).
\vspace{-0.7em}
\item
Define $\tilde{\alpha} = \gcd\!\left\{\frac{nh}{m},\,\frac{ak}{b}\right\},\
p = \frac{nh}{m\tilde{\alpha}},\
q = \frac{ak}{b\tilde{\alpha}}.$
Once Eq.(\ref{main equation 3}) admits a solution of the minimal admissible size, solutions can be constructed for all remaining admissible sizes.
\end{enumerate}
\end{remark}
\begin{example}
 Assume the following tensors $\mathcal{A} \in \mathbb{C}^{2 \times 2 \times 2},$ \ $\mathcal{B}\in \mathbb{C}^{2 \times 1 \times 2}$ and $\mathcal{C} \in \mathbb{C}^{4 \times 3 \times 2}$:
$$\begin{tabular}{c c |c c}
\hline
  \multicolumn{2}{c}{$\mathcal{A}(:,:,1)$} & 
  \multicolumn{2}{c}{$\mathcal{A}(:,:,2)$} \\
  
\hline
1&2&-1&2\\
0&1&0&-1\\
\end{tabular},
\quad
\begin{tabular}{c|c}
\hline
  \multicolumn{1}{c}{$\mathcal{B}(:,:,1)$} & 
  \multicolumn{1}{c}{$\mathcal{B}(:,:,2)$} \\
  
\hline
1&1\\
0&1\\

\end{tabular} \text{and}
\begin{tabular}{c c c|c c c}
\hline
  \multicolumn{3}{c}{$\mathcal{C}(:,:,1)$} & 
  \multicolumn{3}{c}{$\mathcal{C}(:,:,2)$} \\
  \hline
0&11&-5&0&11&-6\\
5&0&11&-7&0&11\\
0&1&0&0&1&0\\
1&0&1&3&0&-1\\

\end{tabular}.$$
By straightforward verification shows that the admissible solution sizes are
$2 \times 3 \times 2,\ 4 \times 6 \times 2,\ 2 \times 3 \times 1 \ \text{and}\ 4 \times 6 \times 1$ then the solution for tensor tensor Eq.(\ref{main equation 3}) is
$$\begin{tabular}{c c c|c c c}
\hline
  \multicolumn{3}{c}{$\mathcal{X}_1(:,:,1)$} & 
  \multicolumn{3}{c}{$\mathcal{X}_1(:,:,2)$} \\
  
\hline
2&1&0&3&5&1\\
1&1&2&-1&-2&3\\
\end{tabular} \in \mathbb{C}^{2 \times 3 \times 2}.$$
\end{example}
Analogous to Theorem \ref{toeplitz theorem}, we characterize the explicit structure of $\mathcal{C}$ using the properties of the semi-tensor product and the dimensional compatibility condition.
\begin{theorem}
Suppose that the tensor tensor Eq.\eqref{main equation 3} is solvable. Then $\mathcal{C}$ admits a block partition in which every block is a Toeplitz tensor. More precisely, $\mathcal{C}$ admits the block
representation
\[
\mathcal{C} =
\left[
\begin{array}{cccc}
\operatorname{Block}_{11}(\mathcal{C}) &
\operatorname{Block}_{12}(\mathcal{C}) &
\cdots &
\operatorname{Block}_{1b}(\mathcal{C}) \\
\vdots & \vdots & \vdots & \vdots \\
\operatorname{Block}_{\frac{h}{m\beta},1}(\mathcal{C}) &
\operatorname{Block}_{\frac{h}{m\beta},2}(\mathcal{C}) &
\cdots &
\operatorname{Block}_{\frac{h}{m\beta},b}(\mathcal{C}) \\
 \hdashline \\
\vdots & \vdots & \vdots & \vdots \\[6pt]
\hdashline \\
\operatorname{Block}_{\frac{h}{\beta}-\frac{h}{m\beta}+1,1}(\mathcal{C}) &
\operatorname{Block}_{\frac{h}{\beta}-\frac{h}{m\beta}+1,2}(\mathcal{C}) &
\cdots &
\operatorname{Block}_{\frac{h}{\beta}-\frac{h}{m\beta}+1,b}(\mathcal{C}) \\[4pt]
\vdots & \vdots & \vdots & \vdots \\
\operatorname{Block}_{\frac{h}{\beta},1}(\mathcal{C}) &
\operatorname{Block}_{\frac{h}{\beta},2}(\mathcal{C}) &
\cdots &
\operatorname{Block}_{\frac{h}{\beta},b}(\mathcal{C})
\end{array}
\right],\]\\
here each $\operatorname{Block}_{j,k}(\mathcal{C})$ is a Toeplitz tensor of identical size of $\frac{bh}{\beta}$  for $j=1, 2, \cdots \frac{h}{\beta}; k=1,2,\cdots b$. In particular, all blocks belong to $\mathbb{C}^{\beta \times \frac{k}{b} \times t}$.
\begin{proof}
Proof follows from matrix case.
\end{proof}
\end{theorem}
\begin{remark}
     Assume that the tensor–tensor Eq.(\ref{main equation 3}) has a solution. If $\mathcal{A, B}$ are F-diagonal tensors with identical diagonal entries, then each sub-tensor of $\mathcal{C}$ exhibits a Circulant tensor structure.
\end{remark}

\section{Computational Performance Analysis}
This section compares the computational performance of the proposed STP-based $t$-product when the Fourier transform is computed using the direct DFT and the FFT. The comparison is presented for both vector and matrix solutions in terms of computational time and theoretical complexity.

The $t$-product can be efficiently computed in the Fourier domain by applying the DFT along the third dimension. For $
\mathcal{A}\in\mathbb{R}^{m\times n\times r}
\ \text{and}\
\mathcal{B}\in\mathbb{R}^{a\times b\times r},$
the Fourier transforms of $\mathcal{A}$ and $\mathcal{B}$ are obtained by applying the transform to their mode-3 fibers. The $t$-product is then computed by multiplying the corresponding transformed frontal slices, followed by the inverse transform. If the DFT matrix is applied explicitly, the transformation of a length-$r$ tube requires $O(r^2)$ operations. Using the FFT reduces this cost to $O(r\log_2 r)$ operations. Therefore, the use of FFT becomes particularly beneficial when the third dimension $r$ is large.

All experiments are performed for tensors
$\mathcal{A}\in\mathbb{R}^{m\times n\times r},
\
\mathcal{B}\in\mathbb{R}^{a\times b\times r}, \
\mathcal{C}\in\mathbb{R}^{h\times k\times r}, $
The third-mode dimension is varied as $r=2,4,8,16,32,64,128.$
For each value of $r$, the DFT and FFT based implementations are
executed under the same computational conditions, and the average
computational time is recorded.

\subsection{Computational Time}
For $\mathcal{A}\in\mathbb{R}^{128\times64\times r}$ and
$\mathcal{B}\in\mathbb{R}^{32\times128\times r}$, we have
$t_1=\operatorname{lcm}(64,32)=64$. Hence, the corresponding
STP expanded tensors are
$\widetilde{\mathcal{A}}\in\mathbb{R}^{128\times64\times r}$ and
$\widetilde{\mathcal{B}}\in\mathbb{R}^{64\times256\times r}$,
resulting in
$\mathcal{C}\in\mathbb{R}^{128\times256\times r}$.
The computational time includes the Fourier transformation, the
slice-wise matrix multiplication in the frequency domain, and the
inverse transformation.

For the direct DFT implementation, the transformation along the third
dimension is performed using an explicit $r\times r$ DFT matrix,
whereas the FFT implementation uses the FFT operation for the same
transformation. Thus, the two methods differ mainly in the
implementation of the Fourier transformation.\\
\begin{table}[h]
\centering
\begin{tabular}{c|c|c}
\hline
$r$ & Average DFT Time (s) & Average FFT Time (s) \\
\hline
2 & 2.01487e-03 & 1.02214e-03 \\
4 & 4.40263e-03 & 2.53396e-03 \\
8 & 8.10507e-03 & 5.09593e-03 \\
16 & 1.42271e-02 & 9.17595e-03 \\
32 & 3.15871e-02 & 2.21874e-02 \\
64 & 9.04443e-02 & 6.31409e-02 \\
128 & 1.70727e-01 & 1.25774e-01 \\
\hline
\end{tabular}
\caption{Average computational time of the DFT and FFT based STP $t$-product (via \texttt{timeit})}.
\end{table}
\begin{figure}[htbp]
    \centering
    \begin{minipage}{0.48\columnwidth}
        \centering
        \includegraphics[width=\linewidth]{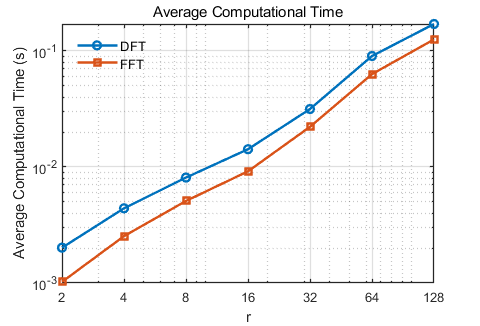}
        \caption*{(a) Average computational time}
    \end{minipage}
    \hfill
    \begin{minipage}{0.48\columnwidth}
        \centering
        \includegraphics[width=\linewidth]{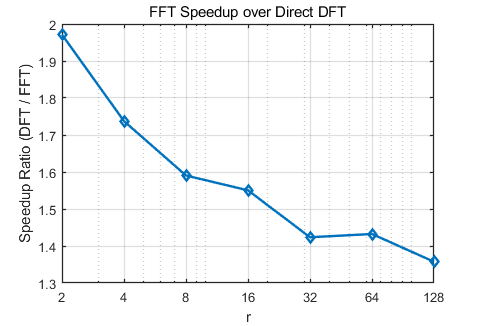}
        \caption*{(b) DFTto-FFT time ratio}
    \end{minipage}
    \label{fig:computational_comparison}
    \caption{Computational time and DFT-to-FFT time ratio for the STP-based $t$-product.}
\end{figure}
 Computational-time plots compare the performance of the direct DFT
and FFT implementations for different values of $r$. As $r$ increases,
the computational time of both methods increases due to the higher cost
of the Fourier transformation and the corresponding frequency-domain
operations. However, the FFT-based implementation generally requires
less computational time than the direct DFT implementation, since FFT
reduces the transformation cost from ${O}(r^2)$ to
${O}(r\log r)$. The absolute computational-time difference generally increases with \(r\), while the FFT maintains a lower computational time in the reported experiments.
\subsubsection{Vector Solution}
For the vector solution, the STP-based tensor equation is converted into a linear system, and the solution is obtained using the Moore-Penrose inverse. The computational time generally increases with the third dimension $r$ due to the computations along the third mode. The experiments are performed with $m=128$, $n=256$, $h=128$, $k=64$, and $a=b=8$, giving $k/b=8$, $n/a=32$, and $p=4$. As shown in Table~\ref{tab:dft_fft_time}, the DFT- and FFT-based implementations have comparable times for small $r$, while the FFT-based implementation becomes faster as $r$ increases. For $r=64$ and $r=128$, the DFT and FFT times are $1.2423$~s and $0.7083$~s, and $2.8027$~s and $1.2822$~s, respectively. This demonstrates the computational advantage of the FFT-based implementation for larger $r$.

\begin{table}[htbp]
\centering
\begin{tabular}{c| c| c}
\hline
$r$ & DFT Average Time (s) & FFT Average Time (s) \\
\hline
2   & $1.03971933\times10^{-2}$ & $9.46328333\times10^{-3}$ \\
4   & $1.74546800\times10^{-2}$ & $1.57026300\times10^{-2}$ \\
8   & $8.95080633\times10^{-2}$ & $4.82243467\times10^{-2}$ \\
16  & $2.16343923\times10^{-1}$ & $1.27823383\times10^{-1}$ \\
32  & $4.72981917\times10^{-1}$ & $2.80739110\times10^{-1}$ \\
64  & $1.24232923\times10^{0}$    & $7.08343740\times10^{-1}$ \\
128 & $2.80265086\times10^{0}$    & $1.28224040\times10^{0}$ \\
\hline
\end{tabular}
\caption{Average computational time for the DFT and FFT based STP vector solution.}
\label{tab:dft_fft_time}
\end{table}

\begin{figure}[htbp]
    \centering
    \begin{minipage}{0.48\columnwidth}
        \centering
        \includegraphics[width=\linewidth]{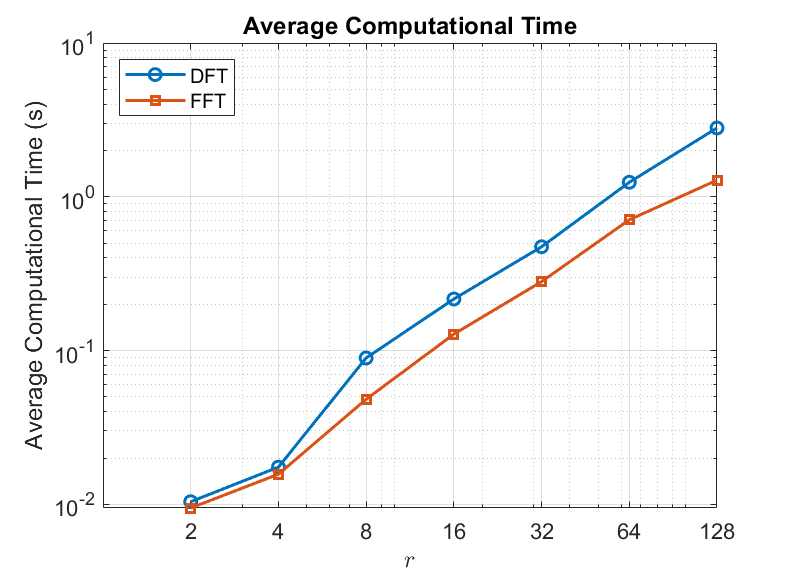}
        \caption*{(a) Average computational time for vector solution}
    \end{minipage}
    \hfill
    \begin{minipage}{0.48\columnwidth}
        \centering
        \includegraphics[width=\linewidth]{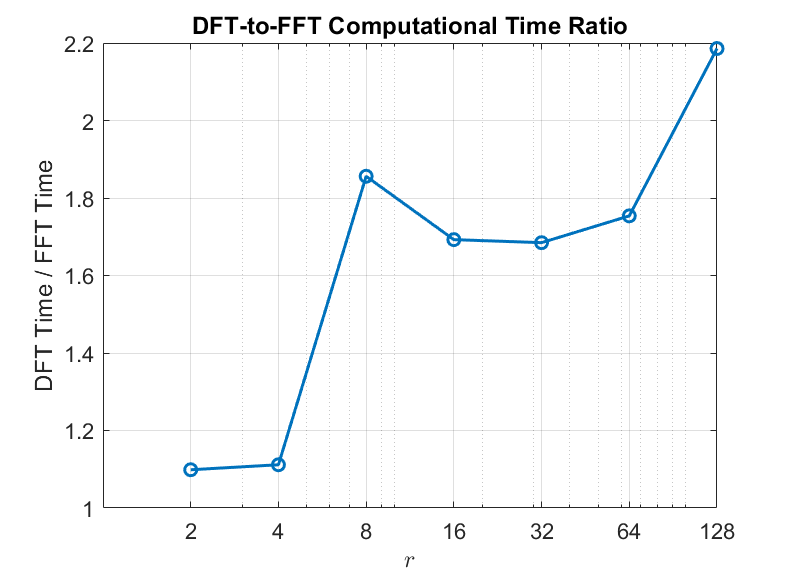}
        \caption*{(b) DFT to FFT time ratio for vector solution}
    \end{minipage}
    \label{fig:computational_comparison_vector}
    \caption{Computational time and DFT-to-FFT time ratio for vector solution.}
\end{figure}

\subsubsection{Matrix Solution}
For the matrix solution, we consider $
X\in\mathbb{R}^{p\times q},\
p=\frac{n}{\beta},\
q=\frac{ak}{b\beta}.$
The parameters used in this experiment are $
m=h=64,\ n=128,\ a=b=8,\ k=32,\
\beta=8,\ p=16,\ q=4.
$
The solution is obtained by constructing the basis responses corresponding to the matrices \(E_{ij}\) and forming the matrix \(M\). The unknown matrix is then computed from $\operatorname{vec}(X)=M^\dagger\operatorname{vec}(C).$

The same procedure is employed for both the direct DFT- and FFT-based implementations of the \(t\)-product.

\begin{table}[htbp]
\centering
\label{tab:matrix-time}
\begin{tabular}{c|c|c}
\hline
$r$ & DFT Time (s) & FFT Time (s) \\
\hline
2   & $1.2452\times10^{-2}$ & $9.8132\times10^{-3}$ \\
4   & $2.1754\times10^{-2}$ & $2.0433\times10^{-2}$ \\
8   & $4.6705\times10^{-2}$ & $4.8800\times10^{-2}$ \\
16  & $2.8217\times10^{-1}$ & $2.4278\times10^{-1}$ \\
32  & $7.6945\times10^{-1}$ & $5.6771\times10^{-1}$ \\
64  & $1.5989\times10^{0}$  & $1.1318\times10^{0}$ \\
128 & $3.9773\times10^{0}$  & $2.3761\times10^{0}$ \\
\hline
\end{tabular}
\caption{Computational time for the matrix-valued solution under direct DFT and FFT implementations.}
\end{table}
\begin{figure}[htbp]
    \centering
    \begin{minipage}{0.48\columnwidth}
        \centering
        \includegraphics[width=\linewidth]{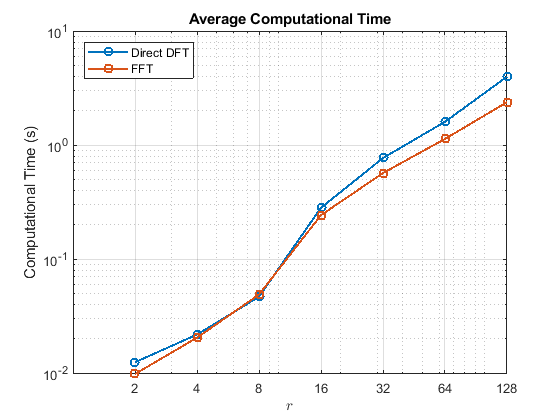}
        \caption*{(a) Average computational time for matrix solution}
        \label{time}
    \end{minipage}
    \hfill
    \begin{minipage}{0.48\columnwidth}
        \centering
        \includegraphics[width=\linewidth]{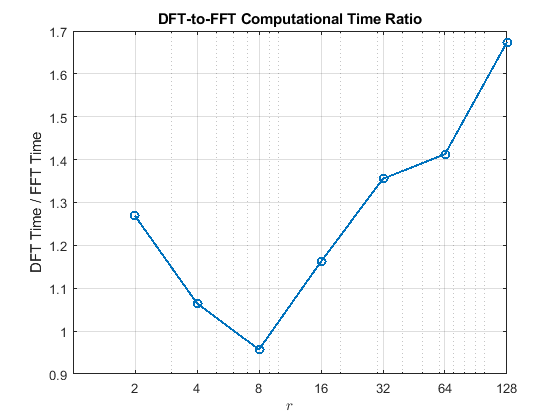}
        \caption*{(b) DFTto-FFT time ratio for matrix solution}
    \end{minipage}
    \label{fig:computational_comparison}
    \caption{Computational time and DFT-to-FFT time ratio for matrix solution.}
\end{figure}
The computational time generally increases with the third-mode dimension \(r\) for both implementations. The FFT-based implementation requires less computational time than the direct DFT-based implementation for most of the tested values of \(r\), with the difference becoming more evident as \(r\) increases. At \(r=128\), the computational time is reduced from \(3.9773\) s for the direct DFT to \(2.3761\) s for the FFT.

\subsection{Comparison of Computational Cost}
The computational cost for obtaining the vector and matrix solutions is
determined by the main operations in the STP-based formulation. Let
$
m_1=m\frac{l_1}{n}, \
n_1=l_1, \
p_1=l_1, \
q_1=q\frac{l_1}{p}.$
Thus,
$\widetilde{\mathcal{A}}\in\mathbb{R}^{m_1\times n_1\times r},
\
\widetilde{\mathcal{B}}\in\mathbb{R}^{p_1\times q_1\times r}.$
In the direct DFT implementation, multiplication by the dense
$r\times r$ DFT matrix requires $O(r^2)$ operations for each tube.
The overall computational cost is
\begin{equation}
\operatorname{Cost}_{\mathrm{DFT}}
=
O\left(
r^2\left[
\frac{ml_1^2}{n}
+\frac{q^2l_1^2}{p}
+\frac{mql_1^2}{np}
\right]
\right)
+
O\left(
\frac{rmql_1^3}{np}
\right).
\label{eq:dft_cost}
\end{equation}

For the FFT implementation, the transformation along the third
dimension requires $O(r\log_2 r)$ operations. Hence, its computational
cost is
\begin{equation}
\operatorname{Cost}_{\mathrm{FFT}}
=
O\left(
r\log_2(r)\left[
\frac{ml_1^2}{n}
+\frac{q^2l_1^2}{p}
+\frac{mql_1^2}{np}
\right]
\right)
+
O\left(
\frac{rmql_1^3}{np}
\right).
\label{eq:fft_cost}
\end{equation}
Thus, the main difference between the two methods is the
transformation cost, which is $O(r^2)$ for the direct DFT and
$O(r\log_2 r)$ for the FFT. The slice-wise multiplication term is
common to both implementations. Therefore, the computational
advantage of FFT becomes more evident as $r$ increases.
We next discuss the computational complexity of the proposed solution
procedures. The detailed derivation of the complexity estimates is
given in Appendix~\ref{app:complexity}.
\subsubsection{For vector and matrix solution} \label{secforapp}
For Algorithm~\ref{alg:vector_solution}, let
$
p=\frac{bn}{ak}, \
S_1=\frac{ml_1^2}{n}+\frac{l_1^2}{p}
+\frac{ml_1^2}{np}, \
K_1=\frac{ml_1^3}{np},$
and let $N=m(k/b)r$. The total computational costs (see Appendix for detail) of the vector
solution using the direct DFT and FFT is
\begin{align}
\operatorname{Cost}_{\mathrm{vector,DFT}}
&=O\left(pr^2S_1+prK_1+Np^2\right), \label{eq6.3}\\
\operatorname{Cost}_{\mathrm{vector,FFT}}
&=O\left(pr\log_2(r)S_1+prK_1+Np^2\right).\label{eq6.4}
\end{align}

Similarly, for Algorithm~\ref{algorithm2}, let
$
p=\frac{n}{\beta}, \
q=\frac{ak}{b\beta},\ 
S_1'=\frac{ml_1^2}{n}+\frac{l_1^2}{p}
+\frac{ml_1^2}{np}, \
K_1'=\frac{ml_1^3}{np}.$
Then the total costs (see Appendix for detail) for the matrix solution are
\begin{align}
\operatorname{Cost}_{\mathrm{matrix,DFT}}
&=O\left(pq\,r^2S_1'
+pq\,rK_1'+hkr(pq)^2\right),\label{eq6.5}\\
\operatorname{Cost}_{\mathrm{matrix,FFT}}
&=O\left(pq\,r\log_2(r)S_1'
+pq\,rK_1'+hkr(pq)^2\right).\label{eq6.6}
\end{align}

These expressions show that the DFT-based implementation has
quadratic dependence on $r$ in the transformation step, whereas the
FFT-based implementation has $O(r\log_2 r)$ dependence. Thus, the
computational advantage of the FFT-based implementation becomes more
evident as $r$ increases.
\begin{figure}[htbp]
    \centering
    \begin{minipage}{0.48\columnwidth}
        \centering
        \includegraphics[width=\linewidth]{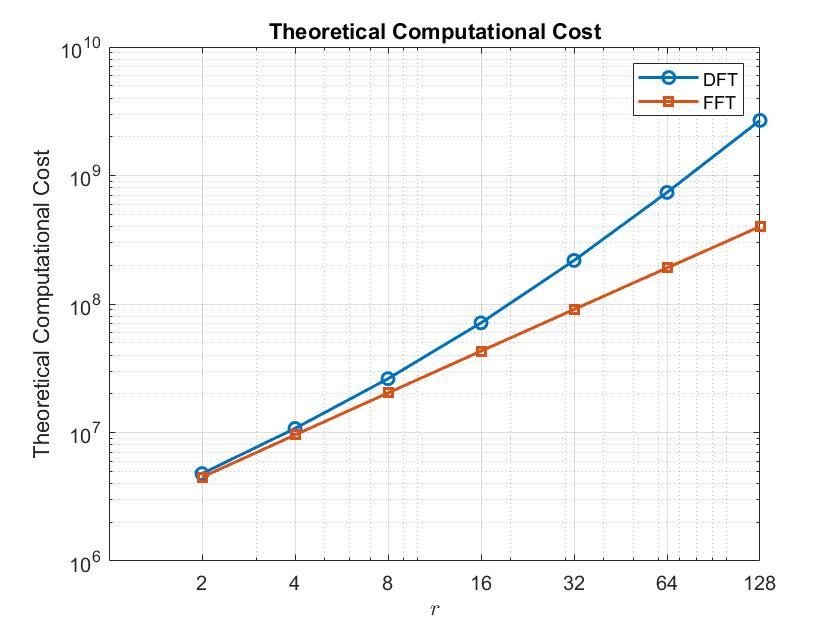}
        \caption*{(a) Computational cost for the DFT- and FFT-based STP
        vector solution.}
        \label{cost_vector}
    \end{minipage}
    \hfill
    \begin{minipage}{0.48\columnwidth}
        \centering
        \includegraphics[width=\linewidth]{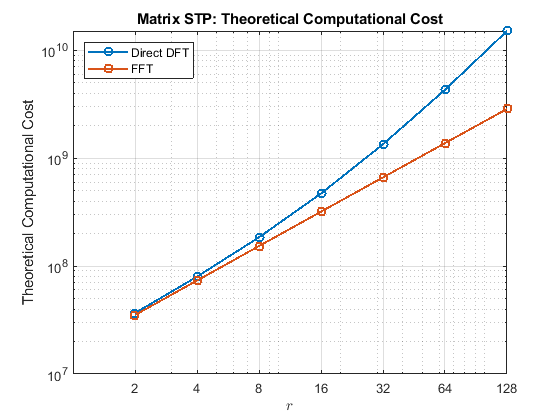}
        \caption*{(b) Computational cost for the DFT- and FFT-based STP
        matrix solution.}
        \label{cost_matrix}
    \end{minipage}
    \caption{Comparison of the computational costs of the DFT- and
    FFT-based implementations.}
    \label{fig:computational_comparison}
\end{figure}
The computational costs of the DFT- and FFT-based implementations
for the vector and matrix solution procedures are compared in
Fig.~\ref{fig:computational_comparison}. In both cases, the FFT-based
implementation requires lower computational cost than the direct DFT,
with the difference becoming more pronounced as $r$ increases.
%%%%%%%%%%%%%%%%%%%%%%%%%%%%%%%%%%%%%%%%%%%%%%%%%%%%%%%%%%%%%%%%%%%%%%%%%%%
\section{Application}

The proposed model is applied to the restoration of blurred and noisy color images. All numerical experiments were implemented in MATLAB R2021b using double-precision arithmetic. We recall from Section~2 that an RGB image is naturally represented as a third-order tensor $\mathcal{X}\in\mathbb{R}^{m\times n\times 3}$, with each frontal slice $\mathcal{X}(:,:,k)$, $k=1,2,3$, storing one color channel. Rather than modeling the cross-channel and spatial blurring separately, we work directly with the tensor equation
$ \mathcal{A} \ltimes \mathcal{X} \ltimes \mathcal{B}=\mathcal{C},$
where $\mathcal{A}$ and $\mathcal{B}$ are third-order blur tensors and $\ltimes$ denotes the semi-tensor product introduced in Section~2, the STP is evaluated using the t-product, which is implemented in the Fourier domain through FFT-based transformation along the third mode. Since the spatial resolution of $\mathcal{A}$ and $\mathcal{B}$ is deliberately taken smaller than that of $\mathcal{X}$ a natural choice when the blur is assumed to act locally and repeat over the image the semi-tensor product resolves the resulting dimension mismatch through Kronecker-product padding, effectively repeating each blur template over $k_{\mathrm{block}}\times k_{\mathrm{block}}$ blocks of the image; we take $k_{\mathrm{block}}=2$ throughout.

\paragraph{Construction of $\mathcal{A}$ and $\mathcal{B}$.} Each of the three frontal slices $\mathcal{A}(:,:,k)$, $k=1,2,3$, is built independently from a Gaussian band matrix,
\[
a_{k\ell} =
\begin{cases}
\dfrac{1}{\sigma\sqrt{2\pi}}\exp\!\left(-\dfrac{(k-\ell)^2}{2\sigma^2}\right), & |k-\ell|\le b,\\[4pt]
0, & \text{otherwise},
\end{cases}
\]
with its own width $\sigma$ and bandwidth $b$, row-normalized so that each slice is row-stochastic, and then scaled by a channel weight $\delta_A(k)$; that is, $\mathcal{A}(:,:,k) = \delta_A(k)\,G(\sigma_A(k),b_A(k))$, and analogously $\mathcal{B}(:,:,k) = \delta_B(k)\,G(\sigma_B(k),b_B(k))$. We use $\sigma_A=\sigma_B=[1,\,1.6,\,2]$, $b_A=b_B=[2,4,6]$, and $\delta_A=\delta_B=[0.8,\,0.1,\,0.1]$. Unlike a fully separable cross-channel blur model, in which every slice would be a scalar multiple of a single common spatial blur matrix, here the three slices of $\mathcal{A}$ (and likewise of $\mathcal{B}$) are three genuinely different Gaussian band matrices, one per channel-mixing component; combined with the t-product's FFT along the third mode, this is what produces the cross-channel mixing, while the dominant $\delta=0.8$ weight keeps each channel primarily blurred by its own template. Zero-mean Gaussian noise with standard deviation $\nu = 5\times10^{-3}$ was added to the blurred data to obtain the observed tensor $\mathcal{C}$.

For the restoration \cite{Bentbib2018, El2021}, both $\mathcal{A}$ and $\mathcal{B}$ are transformed into the Fourier domain along the third mode, and the SVD is computed for each of the three resulting frontal slices. The products of the corresponding left and right singular values then diagonalize the transformed system slice by slice, and a Tikhonov-regularized solution is recovered by filtering each diagonal entry. Rather than fixing one regularization parameter for all three slices, we select a separate parameter $\lambda_i$, $i=1,2,3$, for each slice by generalized cross-validation over $60$ logarithmically spaced candidates in $[10^{-6},1]$; because the three slices carry different combinations of blur severity and noise after the channel-mixing transform, allowing $\lambda_i$ to vary between slices consistently gave better reconstructions than enforcing a single shared value.

We tested the method on two RGB test images. Table~\ref{tab:deblur_results} summarizes the image and blur-template sizes together with the resulting PSNR.

\begin{table}[h]
\centering
\caption{Reconstruction quality for the STP with
t-product deblurring model.}
\begin{tabular}{lc|c|c}
\hline
Image & Size & Blur template size ($\mathcal{A},\mathcal{B}$) & PSNR (dB) \\
\hline
Flower & $524\times524\times3$ & $262\times262\times3$ & $26.89$ \\
Lena   & $224\times224\times3$ & $112\times112\times3$ & $25.87$ \\
\hline
\end{tabular}
\label{tab:deblur_results}
\end{table}

The restored images are shown in Figures~\ref{fig:flower} and \ref{fig:lena}, alongside the original and the observed blurred-and-noisy data.

\begin{figure}[htbp]
    \centering
    \includegraphics[width=0.7\textwidth]{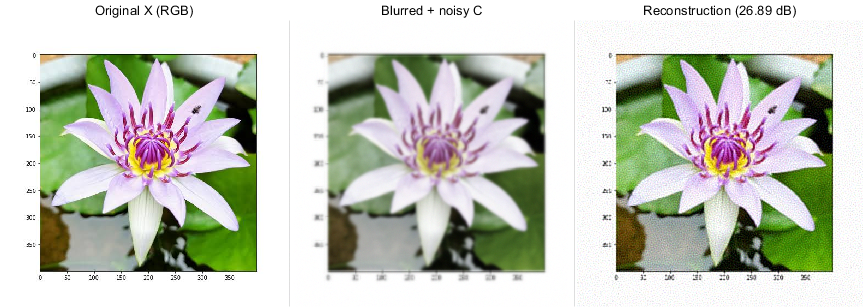}
    \caption{Deblurring result for the flower image.}
    \label{fig:flower}
\end{figure}

\begin{figure}[htbp]
    \centering
    \includegraphics[width=0.7\textwidth]{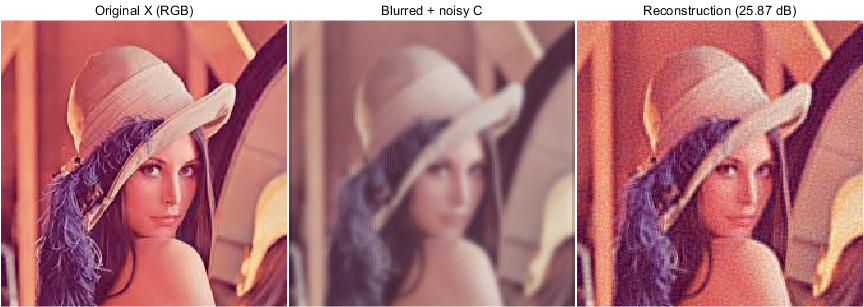}
    \caption{Deblurring result for the Lena image.}
    \label{fig:lena}
\end{figure}
\section{Conclusion}

In this paper, we studied the tensor equation $\mathcal{A}\ltimes \mathcal{X}\ltimes \mathcal{B} = \mathcal{C}$ under the semi-tensor product combined with the $t$-product. We established necessary and sufficient conditions for the existence of vector and matrix solutions, compatibility conditions for tensor solution, characterized the resulting Toeplitz and Circulant structure of $\mathcal{C}$, and developed Moore-Penrose-inverse-based algorithms for computing the solutions. We further analyzed the computational cost of the underlying $t$-product under DFT versus FFT implementation, and demonstrated the framework's effectiveness through a color image deblurring application.

The framework extends naturally to the semi-tensor product combined with the $c$-product; the compatibility conditions carry over, though the explicit solution structure requires further study. As future work, we also intend to explore a generalized semi-tensor product $\mathcal{A}\ltimes_1 \mathcal{B}$ based on averaging rather than identity padding, and to extend the present framework to \emph{interval-valued} tensor equations $\mathcal{A}\ltimes \mathcal{X}\ltimes \mathcal{B} = \mathcal{C}$, where the entries of $\mathcal{A}, \mathcal{B}, \mathcal{C}$ lie in intervals rather than being known exactly. Such a formulation would allow the model to account for uncertainty in the degradation operators for instance, an imprecisely known blur kernel or noise level in the image deblurring application and would require developing interval arithmetic analogues of the semi-tensor and $t$-product operations, along with corresponding solvability and enclosure results for the interval-valued solution set.

\section*{Conflicts of interest}
The authors declare that they have no conflict of interest.

 \section*{Data Availability Statements}
 No new datasets were generated in this study.

\section*{Appendix}
\label{app:complexity}

For completeness, we give the details of the complexity analysis presented
in Section~\ref{secforapp}. We first consider Algorithm~\ref{alg:vector_solution}.
The matrix $M$ is constructed using the columns
$V_L(\widetilde{\mathcal A}_t)$, $t=1,\ldots,p$. The
complexity of one evaluation can be obtained by taking $q=1$ in
\eqref{eq:dft_cost}--\eqref{eq:fft_cost}. For the direct DFT, this gives
\begin{equation}
O(r^2S_1)+O(rK_1),
\label{eq:alg1_single_dft}
\end{equation}
while for the FFT, the cost is
\begin{equation}
O(r\log_2 r,S_1)+O(rK_1).
\label{eq:alg1_single_fft}
\end{equation}

There are $
p=\frac{bn}{ak}
$ such evaluations in the construction of $M$. Therefore, the total cost
of constructing $M$ is
$O(pr^2S_1)+O(prK_1),$ for the direct DFT
 $O(pr\log_2 r,S_1)+O(prK_1$ for the FFT. 

Now, $M\in\mathbb C^{N\times p}$, where
$
N=m\left(\frac{k}{b}\right)r.
$

The Moore-Penrose inverse $M^\dagger$ is computed using the SVD, whose
cost is $O(Np^2)$. Thus, by combining the cost of constructing $M$ with
the cost of computing $M^\dagger$, we obtain the total complexities given
in \eqref{eq6.3} and~\eqref{eq6.4}.

We next consider Algorithm~\ref{algorithm2}. In this case,
$$
p=\frac{n}{\beta},
q=\frac{ak}{b\beta}.
$$
The matrix $N$ is constructed using $pq$ basis-matrix evaluations. For
one such evaluation, the direct DFT requires
\begin{equation}
O(r^2S_1')+O(rK_1'),
\label{eq:alg2_single_dft}
\end{equation}
whereas the FFT requires
\begin{equation}
O(r\log_2 r,S_1')+O(rK_1').
\label{eq:alg2_single_fft}
\end{equation}

Since $pq$ evaluations are needed, the corresponding costs for
constructing $N$ are
$O(pq,r^2S_1')+O(pq,rK_1'),
 \text{for the direct DFT},\ 
O(pq,r\log_2 r,S_1')+O(pq,rK_1'),
 \text{for the FFT}.$

Finally, since $
N\in\mathbb C^{hkr\times pq},
$ the SVD-based computation of $N^\dagger$ requires
$O\bigl(hkr(pq)^2\bigr)$ operations. Adding this cost to the cost of
constructing $N$ gives the total complexities stated in
\eqref{eq6.5} and~\eqref{eq6.6}.

\end{document}